\documentclass[12pt, final]{colt2019}

\usepackage{jmlrutils}
\usepackage{times}

\usepackage[utf8]{inputenc} 
\usepackage[T1]{fontenc}
\usepackage{graphicx,,psfrag}
\usepackage{amssymb,amsmath,enumerate}
\usepackage{makecell}

\newtheorem*{Proposition*}{Proposition}

\newcommand{\EE}{{\mathbf{E}}}

\newcommand{\dE}{\mathbb {E}}
\newcommand{\dP}{\mathbb{P}}

\newcommand{\cF}{\mathcal {F}}

\newcommand{\cE}{\mathcal {E}}
\newcommand{\cA}{\mathcal {A}}
\newcommand{\cG}{\mathcal {G}}

\newcommand{\cK}{\mathcal {K}}

\newcommand{\Poi}{ \mathrm{Poi}}

\DeclareMathOperator{\ov}{ov}

\newcommand{\II}{1\!\!{\sf I}}

\newcommand{\BEAS}{\begin{eqnarray*}}
\newcommand{\EEAS}{\end{eqnarray*}}
\newcommand{\BEA}{\begin{eqnarray}}
\newcommand{\EEA}{\end{eqnarray}}
\newcommand{\BEQ}{\begin{equation}}
\newcommand{\EEQ}{\end{equation}}
\newcommand{\BIT}{\begin{itemize}}
\newcommand{\EIT}{\end{itemize}}
\newcommand{\BNUM}{\begin{enumerate}}
\newcommand{\ENUM}{\end{enumerate}}

\title{Planting trees in graphs, and finding them back}

\coltauthor{%
  \Name{Laurent Massoulié} \Email{laurent.massoulie@inria.fr}\\
  \addr Inria, \'Ecole Normale Supérieure, PSL Research University, Microsoft Research-Inria Joint Centre
  \AND
  \Name{Ludovic Stephan} \Email{ludovic.stephan@inria.fr}\\
  \addr Inria, \'Ecole Normale Supérieure, PSL Research University
  \AND
  \Name{Don Towsley} \Email{towsley@cs.umass.edu}\\
  \addr{University of Massachusetts Amherst}
}

\begin{document}

\maketitle
\begin{abstract}
In this paper we study the two inference problems of detection and reconstruction in the context of planted structures in sparse Erd\H{o}s-Rényi random graphs $\mathcal G(n,\lambda/n)$ with fixed average degree $\lambda>0$. Motivated by a problem of communication security, we focus on the case where the planted structure consists in the addition of a tree graph. 

In the case of planted line graphs, we establish the following phase diagram for detection and reconstruction. In a low density region where the average degree $\lambda$ of the original graph is below some critical value $\lambda_c=1$, both detection and reconstruction go from  impossible to easy as the line length $K$ crosses some critical value $K^*=\ln(n)/\ln(1/\lambda)$, where $n$ is the number of nodes in the graph. In a high density region where $\lambda>\lambda_c$, detection goes from impossible to easy as $K$ goes from $o(\sqrt{n})$ to $\omega(\sqrt{n})$. In contrast, reconstruction remains impossible so long as $K=o(n)$. 

We then consider planted $D$-ary trees of varying depth $h$ and $2\le D\le O(1)$. For these we identify a low-density region $\lambda<\lambda_D$, where $\lambda_D$ is the threshold for emergence of the $D$-core in Erd\H{o}s-Rényi random graphs $\mathcal G(n,\lambda/n)$ for which the following holds. There is a threshold $h*=g(D)\ln(\ln(n))$ with the following properties. Detection goes from impossible to feasible as $h$ crosses $h*$. Interestingly, we show that only partial reconstruction is feasible at best for $h\ge h*$. We conjecture a similar picture to hold for $D$-ary trees as for lines in the high-density region $\lambda>\lambda_D$, but confirm only the following part of this picture: Detection is easy for $D$-ary trees of size $\omega(\sqrt{n})$, while at best only partial reconstruction is feasible for $D$-ary trees of any size $o(n)$.

These results provide a clear contrast with the corresponding picture for detection and reconstruction of {\em low rank} planted structures, such as dense subgraphs and block communities.  In the examples we study, there is i) an absence of hard phases for both detection and reconstruction, and ii) a discrepancy
between detection and reconstruction, the latter being impossible for a wide range of parameters where detection is easy. The latter property does not hold for previously studied low rank planted structures.

\end{abstract}
\newpage
\pagenumbering{arabic}
\section{Introduction}
This paper is concerned with the detection of additional structures planted in a graph initially without structure (such as an Erd\H{o}s-Rényi graph) and, in case such a structure is detected, with the reconstruction of the corresponding structure. We focus on planted structures that consist in a superimposed graph, and more specifically on superimposed trees. 

A first motivation for this focus stems from the following application scenario. Assume that the original graph without planted structure represents normal communications among agents, while the superimposed graph represents communications among a subset of {\em attackers} who, when active, connect directly among themselves to coordinate their activity. Detection then amounts to estimating whether an attack occurs, while reconstruction amounts to identifying the attackers in case of an attack. 

A second motivation is theoretical: previous work reviewed in Section \ref{sec:2} has shown that detection and reconstruction of planted structures in graphs displays rich and intriguing behaviour, with phases where the task is either impossible, computationally hard, or easy. It is important to understand what causes such phases, and whether phases for detection always coincide with the corresponding phase for reconstruction. Our present study sheds light on these questions, by showing that in the cases of planted tree structures we consider, no hard phase occurs, while feasibility phases of detection and reconstruction differ widely.  In contrast, the latter property does not hold for previously studied low rank planted structures.

More specifically, our contributions are as follows. In the particular case of planted line graphs, we determine the complete phase diagram for detection and reconstruction: In a low density region where the average degree $\lambda$ of the original graph is below some critical value $\lambda_c$, both detection and reconstruction go from  impossible to easy as the line length $K$ crosses some critical value $K^*=f(\lambda)\ln(n)$, where $n$ is the number of nodes in the graph. In a high density region where $\lambda>\lambda_c$, detection goes from impossible to easy as $K$ goes from $o(\sqrt{n})$ to $\omega(\sqrt{n})$. In contrast, reconstruction remains impossible so long as $K=o(n)$. 

We then consider the case of $D$-ary trees for fixed $D>1$, of height $h$. For these our results provide a similar picture with significant differences. Specifically, there exists a limit height $h_* = \ln\ln(D) + O(1)$ such that detection is impossible if $h < h_* - \ln(h_*)$, and easy for $h > h_* + \Omega(1) $. In that latter case, non-trivial reconstruction is feasible, but it must fail on a non-vanishing fraction of the $K$ attack nodes. In a high-density region $\lambda>\lambda_D$, we have again that detection is easy for $K=\omega(\sqrt{n})$, and that reconstruction must fail at least on a fraction of nodes. 

The paper is organized as follows. We review related work in Section \ref{sec:2}. We describe our model and main results in Section \ref{sec:3}. The proofs for planted lines and planted $D$-ary trees are in Sections \ref{sec:4} and \ref{sec:5} respectively, with detailed proofs of auxiliary results in the Appendix. 

\section{Related work}\label{sec:2}

Planted clique detection and reconstruction has been the object of many works, see e.g. \cite{dekel_gurel-gurevich_peres_2014}, \cite{Deshpande12}, \cite{barak16} for recent results and surveys. A central result in that context is that detection appears hard (i.e. no algorithm is known to succeed at detection in polynomial time) for cliques of size $o(\sqrt{n})$ planted in $G(n,1/2)$.  IT thresholds for planted dense subgraph detection are developed in \cite{verzelen2015}.

Computational hardness of planted clique is used  in reduction arguments to show that other planted structure detection problems are hard, eg sparse PCA \cite{berthet2013lowerSparsePCA_arxiv}, and dense subgraph detection \cite{HajekWuXu14}. The latter also displays IT-impossible phases, hard phases and easy phases. A systematic development of such reductions between problems with planted structure is initiated in \cite{pmlr-v75-brennan18a}. 

Community detection and reconstruction has also been thoroughly studied, the seminal article \cite{Decelle11} introducing several conjectures on feasibility of detection and reconstruction for the stochastic block model. Almost all conjectures in \cite{Decelle11} have been verified in subsequent works, in particular \cite{Mossel12}, \cite{Massoulie13}, \cite{Mossel13}, \cite{abbe16}. 

Presence of specific subgraphs in random graphs has been thoroughly studied, see e.g. \cite{janson11}. We leverage the corresponding techniques in our study of low density regions, for which detection feasibility corresponds to absence of copies of the planted graph structure in the original random graph. 

Most planted structures considered so far were typically of ``low rank'' (e.g. planted dense graph's expected adjacency matrix is, up to diagonal terms, a rank one perturbation); in contrast, adjacency matrices of trees and lines are not close to a low rank matrix.  One notable exception is the planted Hamiltonian cycle reconstruction addressed in \cite{DBLP:journals/corr/abs-1804-05436}. 

\section{Model and main results}\label{sec:3}

A total population of $n$ agents interconnects according to one of the following two modalities. Under the null hypothesis $H_0$ the interconnection does not display any specific structure. We assume that the corresponding graph $G$ is an Erd\H{o}s-Rényi $\cG(n,p)$ graph, with edge probability $p\in[0,1]$  taken equal to $\lambda/n$ for some fixed $\lambda>0$. We thus focus on sparse random graphs with average degree $O(1)$. 
Under the alternative hypothesis $H_1$, the graph $G$ is the union of a base graph $G_0$ distributed according to $\cG(n,p)$, with another graph $G'$ connecting a distinguished subset $\cK$ of nodes. Specifically, for a fixed graph $\Gamma$ on node set $[K]$ with edge set $\cE$, and an injective map $\sigma:[K]\to [n]$ chosen uniformly at random and independently of $G_0$, $G'$ consists of the nodes $\cK=\{\sigma(i),i\in[K]\}$ and edges $\{(\sigma(i),\sigma(j)),(i,j)\in\cE\}$. 

We shall mostly focus on tree graphs $\Gamma$, and more specifically on $D$-ary trees, i.e. trees with a distinguished root, or depth-0 node, and for each $\ell\in[h-1]$, $D^{\ell}$ depth-$\ell$ nodes being connected to one parent at depth $\ell-1$ and $D$ children at depth $\ell+1$. The two exteme cases are a line graph for $D=1$ and a star for $D=K-1$.

We are interested in answering, on the basis of an observed graph $G$, the following questions: 

{\bf Q1 (Detection)}: For a given planted graph shape $\Gamma$ (e.g. line, star, $D$-ary tree,$\ldots$), under what parameter regimes specified by $\lambda$ and $K$ is there a test that distinguishes $H_0$ from $H_1$ with error probabilities of both kinds going to zero as $n\to\infty$?
This is an information-theoretic property characterized by the likelihood ratio $\frac{\dP_1(G)}{\dP_0(G)}$, where $\dP_i$ denotes the distribution of $G$ uner $H_i$, $i=\{0,1\}$. Indeed
by the Neyman-Pearson lemma, among tests with given probability of correctly deciding $H_1$, there is one which minimizes probability of erroneously rejecting $H_0$ which decides $H_1$ if and only if the likelihood ratio  $L(G):=\frac{\dP_1(G)}{\dP_0(G)}$ is larger than some threshold $\tau$. We can ask the same question as Q1 when we restrict ourselves to tests that can be implemented in polynomial time. This then corresponds to a computational property. 

{\bf Q2 (Reconstruction)}: Can one reconstruct the planted structure $G'$, or at least a subset of its constituent nodes? Several metrics of reconstruction accuracy are possible. We shall focus on the following {\bf overlap} metric, which we now define for estimation procedures that produce a set  $\hat{\cK}$ of $K$ nodes in $[n]$, aimed to estimate at best the actual set $\cK$ of $K$ nodes involved in the attack. 
\begin{definition}
The {\bf overlap} of a set $\hat{\cK}$ estimating the actual ground truth $\cK$ is by definition the expected size of their intersection, i.e.
$$
\ov(\hat{\cK}):=\sum_{i\in[n]}\dP(i\in\hat{\cK}\cap\cK).
$$
We say that a particular reconstruction $\hat{\cK}$ of size $K$ {\bf fails} if $\ov(\hat{\cK})=o(K)$, {\bf succeeds} if $\ov(\hat{\cK})=K(1-o(1))$, and {\bf partially succeeds} if $\ov(\hat{\cK})=cK(1-o(1))$ for some $c\in(0,1)$.
\end{definition}
Reconstruction (respectively, partial reconstruction) is then deemed {\bf feasible} if there exists an  estimator $\hat{\cK}$ that is successful (respectively, partially successful). These properties are of an information-theoretic nature. Indeed the best possible overlap is achieved by the so-called Maximum a Posteriori (MAP) estimation procedure, and these properties are therefore determined by the overlap of the MAP estimator. One can, as for detection, consider a computational version of reconstruction: reconstruction (respectively, partial reconstruction) is {\bf easy} when it can be achieved by an estimator $\hat{\cK}$ that is efficiently computable. 

Before stating our results for planted lines and $D$-ary trees, we first consider planted star graphs, for which a simpler picture holds:
\begin{theorem}\label{thm:star}
For any fixed $\lambda>0$, a planted star of size $K=\ln(n)/\ln(\ln(n))[1-\omega(1/\ln(\ln(n)))]$ is not detectable, while both detection and reconstruction of a planted star of size $K=\ln(n)/\ln(\ln(n))[1+\omega(1/\ln(\ln(n)))]$ are easy.
\end{theorem}


\begin{table}[htb]
    \centering
    \def\arraystretch{1.2}
    \subfigure[Subcritical regime : $\lambda < 1$]{%
        \begin{tabular}{|c|c|}
            \hline
            $K < \ln(n)/\ln(1/\lambda)$ & $\ln(n)/\ln(1/\lambda) < K \ll n/\ln(n)$ \\
            \hline
            \makecell{Detection \& reconstruction \\ IT impossible} & \makecell{Detection \& reconstruction \\ easy} \\
            \hline
        \end{tabular}%
    }
    
    \vspace{1em}
    
    \subfigure[Supercritical regime : $\lambda > 1$]{%
        \begin{tabular}{|c|c|}
            \hline
            $K \ll \sqrt{n}$ & $\sqrt{n} \ll K \ll n$ \\
            \hline
            \makecell{Detection \& reconstruction \\ IT impossible} & \makecell{Detection easy \\ reconstruction IT impossible} \\
            \hline
        \end{tabular}%
    }
    \caption{Summary of results for planted line graph}
    \label{tab1}
\end{table}
The result for line graphs, summarized in Table \ref{tab1}, is
\begin{theorem}[Line graphs]
In the low-density region $\lambda<\lambda_c=1$, detection and reconstruction are impossible if $K=\ln(n)/\ln(1/\lambda)-\omega(\ln(\ln(n)))$, while both detection and reconstruction are easy if $K=\ln(n)/\ln(1/\lambda)+\omega(1)$ and $K=o(n/\ln(n))$.

In the high-density region $\lambda>\lambda_c=1$, detection and reconstruction are impossible if $K=o(\sqrt{n})$, detection is easy if $K=\omega(\sqrt{n})$, while reconstruction is impossible for $K=o(n)$. 
\end{theorem}

\begin{table}[htb]
    \centering
    \def\arraystretch{1.2}
    \subfigure[Subcritical regime : $\lambda < \lambda_D$]{%
        \begin{tabular}{|c|c|}
            \hline
            $h < \log_D\log_D(n)$ & $\log_D\log_D(n) < h < \log_D(n)$ \\
            \hline
            \makecell{Detection \& reconstruction \\ IT impossible} & \makecell{Detection easy \\ complete reconstruction impossible} \\
            \hline
        \end{tabular}%
    }
    
    \vspace{1em}
    
    \subfigure[Supercritical regime : $\lambda > \lambda_D$]{%
        \begin{tabular}{|c|c|}
            \hline
            $ h < \log_D(n)/2 $ & $\log_D(n)/2 < h < \log_D(n)$ \\
            \hline
            \makecell{Detection unknown \\ complete reconstruction impossible} & \makecell{Detection easy \\ complete reconstruction impossible} \\
            \hline
        \end{tabular}%
    }
    \caption{Summary of results for planted $D$-ary tree}
    \label{tab2}
\end{table}

For $D$-ary trees, the results are similar. However the critical parameter $\lambda_D$ defined in \eqref{eq:def_lambda_D} is the threshold for emergence of the $D$-core (see \cite{Moore:2011:NC:2086753}), and only partial reconstruction is possible in the subcritical regime $\lambda<\lambda_D$. We consider $D$-ary trees $\Gamma$ of depth $h$ with corresponding size $K=\frac{D^{h+1}-1}{D-1}$ ; the main results (in terms of $h$) are summarized in Table~\ref{tab2}.

\begin{theorem}[D-ary trees]
 In the low-density region $\lambda < \lambda_D$, there exist two parameters $\underline h$ and $\bar h$ such that the following holds.

$\bar h = \ln\ln(n)/\ln(D) + \Theta(1)$, and $\underline h = \bar h - 1$ for almost all $\lambda$.
    
When $h \leq \underline h - O(\ln(\underline h))$, both detection and reconstruction are impossible with high probability.

Detection is easy whenever $h \geq \bar h + O(1)$.

\noindent For any $\lambda>0$, hence in both low-density and high-density regions, detection is easy whenever $K = \omega(\sqrt{n})$ while complete reconstruction is impossible for $K = o(n)$.
\end{theorem}
\section{Preliminary results}

We now state three results that hold for arbitrary planted structures, and that will be used extensively. The first is a characterization of the likelihood ratio $\frac{\dP_1}{\dP_0}$:
\begin{lemma}\label{lem:likelihood_ratio}
The likelihood ratio $L(G)=\frac{\dP_1(G)}{\dP_0(G)}$ is given by $L(G)=\frac{X_{\Gamma}}{\dE_0(X_{\Gamma})}$, where $X_{\Gamma}$ denotes the number of copies of $\Gamma$ in $G$.
\end{lemma}
The second gives a generic detection process that succeeds for $K$ large enough, and all planted graph structures $\Gamma$ that are connected.
\begin{theorem}\label{thm:easy_detect}
  Assume that $\lambda>0$, $K=\omega(\sqrt{n})$, and the hidden graph is any connected subgraph on $K$ nodes, not necessarily a line.
  Then the total variation distance $|\dP_1-\dP_0|_{var}$ between $\dP_0$ and $\dP_1$ goes to 1 as $n\to \infty$. 
  
  Let $A_i$, $i\in\{1,2,3\}$ denote the number of size $i$-connected components in $G$, $\hat{\lambda}=(n A_3)/(A_1 A_2)$, and $\hat{k}=n-e^{\hat{\lambda}}A_1$. The test that decides $H_1$ if $\hat{k}\ge t_n:=\sqrt{K\sqrt{n}}$, and $H_0$ otherwise is polynomial-time computable and distinguishes with high probability graphs sampled from $\dP_1$ or $\dP_0$.
\end{theorem}

\begin{remark}
  When $\lambda$ is known, a simpler test based on the number of edges in the graph  also succeeds. The test in Theorem \ref{thm:easy_detect} still applies even when $\lambda$ is unknown. The proof further implies that under $\dP_1$, $G$ can be distinguished from $\cG(n,\lambda'/n)$ for any $\lambda'$ not necessarily equal to $\lambda$.
\end{remark}

Finally, it is important to note that, as evidenced in \cite{BanksMoore18}, impossibility of detection does not imply immediately that of reconstruction. Fortunately, in our setting, the following result will imply the latter as soon as the former is proved :

\begin{theorem}\label{thm:reconstruction_fails}
  Assume that $K = o(\sqrt{n})$ that $\dE_0\left(X_\Gamma\right) = \omega(1)$ and that $\dE_0\left(L^2\right) = 1 + o(1)$. Then, for every estimator $\hat\cK$ of the planted set $\cK$, we have
  $$ \mathrm{ov}(\hat\cK) = o(K), $$
  that is, reconstruction fails as well.
\end{theorem}
\section{Proof strategy for planted paths}\label{sec:4}
We say that the ordered set $\{i_1,\ldots,i_K\}$ of $K$ distinct  nodes in $[n]$ is a $K$-path in $G$ if and only if the edges $(i_{\ell},i_{\ell+1})$ are present in $G$ for all $\ell=1,\ldots, K-1$. The previous Lemma \ref{lem:likelihood_ratio} yields, in the case where $\Gamma$ is the line graph, the following result, whose proof is in the appendix:
\begin{lemma}\label{lemma:LR_paths}
For planted $K$-path, the likelihood ratio reads
\begin{equation}\label{LR_ll}
L(G):=\frac{\dP_1(G)}{\dP_{0}(G)}=\frac{1}{n(n-1)\cdots(n-K+1)}|\{\hbox{$K$-paths in $G$}\}|\left(\frac{\lambda}{n}\right)^{-K+1}\cdot
\end{equation}
Moreover one has
\begin{equation}\label{LR_lll}
\dE_0(L^2)=\dE_0(x^S),
\end{equation}
where $x=n/\lambda$, and $S$ is a random variable counting the number of edges common to the $K$-path $(1-2-\cdots-K)$ and a random $K$-path $\pi$ chosen uniformly at random among the $n(n-1)\cdots(n-K+1)$ possible ones on node set $[n]$.
\end{lemma}
%
%
\subsection{Impossibility of detection}
We have the following
\begin{theorem}\label{thm:detection_fails}
Assume that $\lambda>1$ and $K=o(\sqrt{n})$, or alternatively that $\lambda<1$ and $K=\ln(n)/\ln(1/\lambda)-\omega(\ln(\ln(n)))$. Then the total variation distance $|\dP_1-\dP_0|_{var}$ between $\dP_0$ and $\dP_1$ goes to zero as $n\to \infty$. Thus for any arbitrary test $T(G)\in\{0,1\}$, $\dP_1(T(G)=1)-\dP_0(T(G)=1)\to 0$ as $n\to \infty$.
\end{theorem}
By a standard argument, the variation distance $|\dP_1-\dP_0|_{var}$ is upper-bounded by$\sqrt{\dE_0(L^2)-1}$, and thus the Theorem is a direct consequence of the following
\begin{lemma}\label{lemma_1}
Assume that $\lambda>1$ and $K=o(\sqrt{n})$, or alternatively that $\lambda<1$ and $K=\ln(n)/\ln(1/\lambda)-\omega(\ln(\ln(n)))$. Then $\lim_{n\to\infty}\dE_0(L^2)=1$.
\end{lemma}
The proof of Lemma \ref{lemma_1} (details in the Appendix)
is based on an analysis of  expression \eqref{LR_lll}.  Set $Z_t=1$ if edge $(I_{t},I_{t+1})$ is part of path $(1\cdots K)$, $Z_t=0$ if it is not part of that path, but $I_{t+1}\in[K]$, and finally $Z_t=-1$ if $I_{t+1}\notin [K]$, so that
\begin{equation}\label{squareLR}
\dE_0(L^2)=\dE_0(x^{\sum_{t=1}^{K-1}Z_t^+})
\end{equation}
In order to upper-bound this expression, a key step is the following Lemma, which exhibits a tractable upper bound involving a Markov chain:
\begin{lemma}\label{lemma:mkv_upper_bound}
Let $n':=n-K$. The Markov chain $\{Z'_t\}_{t\ge 1}$ taking values in $\{-1,0,1\}$ with transition probability matrix
\begin{equation}\label{eq:transition_proba_Z'}
P:=\left(
\begin{array}{lll}
1-K/n' & K/n'& 0\\
1-K/n' & (K-2)/n' & 2/n'\\
1-K/n' & (K-1)/n' & 1/n'
\end{array}
\right)
\end{equation}
can be constructed jointly with process $\{Z_t\}_{t\ge 1}$ so that, for all $m\ge 1$, one has
\begin{equation}\label{eq:coupling1}
\dE_0(x^{\sum_{t=1}^{m}Z_t^+})\le \dE_0(x^{\sum_{t=1}^{m}Z'^+_t}).
\end{equation}
\end{lemma}
Its proof is in the appendix, together with the analysis of the right-hand side of \eqref{eq:coupling1}. The latter relies on spectral analysis of a matrix derived from $P$ in \eqref{eq:transition_proba_Z'}, which leverages perturbation arguments as $K/n\to 0$. It concludes the proof of Lemma \ref{lemma_1} by showing that $\dE_0(L^2)=1+o(1)$ under the Lemma's assumptions. 

\subsection{Easiness of detection and reconstruction, sparse case}
Assume $\lambda<1$ and $K=\ln(n)/\ln(1/\lambda)+\omega(1)$. Detection is then easy: under $\dP_0$, the expected number of $K$-paths in the graph is $o(1)$. A test which decides $\dP_1$ if there is a $K$-path and $\dP_0$ otherwise thus discriminates the two hypotheses with high probability. Presence of a $K$-path can moreover be determined in polynomial time by running depth-first searches from each node in $G$.

For reconstruction, we need the following
\begin{lemma}\label{lemma:reconstr_lines_subcritical}
For $\lambda<1$, $K=\ln(n)/\ln(1/\lambda)+\omega(1)$ and $K=o(n/\ln(n))$, let $C$ be the connected component of the graph containing the longest path. Apply $\sqrt{K}$ times a {\em peeling operation} to $C$, which consists in removing all degree one nodes, to obtain set $C'$. Under $\dP_1$, set $C'$ and its intersection with the planted path both have with high probability size $K\pm o(K)$.
\end{lemma}
The Lemma readily implies a polynomial-time algorithm for reconstruction that achieves overlap $K-o(K)$: set $C'$ can be obtained in polynomial time. By adding / removing $o(K)$ nodes to it one obtains a set of size $K$ with overlap $K-o(K)$.
\subsection{Impossibility of reconstruction, dense case}
We assume $\lambda>1$ and $K=\omega(\sqrt{n})$. We have seen that with high probability, observation of $G$ allows to determine whether or not an attack has taken place. We now assume that an attack has indeed happened. We have the following result, showing the impossibility of efficient planted structure reconstruction:
\begin{theorem}\label{thm:impossible_reconstr_line}
Given $\lambda>1$, $K=\omega(\sqrt{n})$, $K=o(n)$, and a realization $G$ of the graph under $\dP_1$, any estimator $\hat{\cK}$ of the ground truth achieves negligible overlap, i.e. $\ov(\hat\cK)=o(K)$.
\end{theorem} 
Its proof structure is as follows. Fix an arbitrary integer $\tau\ge 1$. We shall establish that necessarily 
\begin{equation}\label{overlap_control_1}
\ov(\cK)\le K/(\tau+1)+o(K). 
\end{equation}
Fix 
\begin{equation}\label{eq:LD_1}
L=C\ln(n)\hbox{ for some suitable constant $C$, } D\gg L\hbox{ and } D^2\ll \frac{n}{\ln(n)}.
\end{equation}  
Condition on the event the attack path is precisely $k_1,\ldots,k_K=:k_1^K$. Chop the attack path into $K/(L+D)$ contiguous segments, each of length $M:=L+D$. 

Consider the $\ell$-th segment $\{k_{(\ell-1)M+1},\ldots,k_{\ell M}\}$. We shall construct, for some $I(\ell)\in[(\ell-1)M+1,(\ell-1)M+L]$, $\tau$ random paths of edges in the graph $G$ of the form $k_{I(\ell)},I_2(t,\ell),I_3(t,\ell),\ldots,I_{D}(t,\ell),k_{I(\ell)+D}$ for $t\in[\tau]$ such that the nodes $I_2(t,\ell),\ldots,I_{D}(t,\ell)$ are all distinct, none of them belongs to the attack path, and such that the paths $(k_1,\ldots,k_K)=:k_1^K$ and $k_1^{I(\ell)},I_2^{D}(t,\ell),k_{I(\ell)+D}^K$ are statistically indistinguishable. More precisely, we have the following:
\begin{lemma}
\label{dvar_control}
There is a construction, for any $\ell\in[K/M]$, of $\tau$ random paths 
$$
k_{I(\ell)},I_2(t,\ell),I_3(t,\ell),\ldots,I_{D}(t,\ell),k_{I(\ell)+D}, \;t\in[\tau],
$$
such that for any $i\in[(\ell-1)M+1,(\ell-1)M+L]$, any $\tau$ disjoint ordered sets of $D-1$ distinct nodes $i_2^{D}(t)$, $t\in[\tau]$ in $[n]\setminus k_1^K$, we have
\begin{equation}\label{eq_dvar_1}
d_{var}(\dP_1(G\in\cdot |\cK=k_1^K,I(\ell)=i,(I_2^{D}(t,\ell))_t=(i_2^{D}(t))_t),\dP_0(G\in \cdot|k_1^K\in G, (k_i,i_2^{D}(t),k_{i+D})_t\in G))=\epsilon=o(1).
\end{equation}
This construction moreover verifies the following property. There is an event $\cE$ such that $\dP_1(\cE)=1-o(1)$, and such that, denoting $|(\cup_{t\in[\tau]}I_2^{D}(t,\ell))\cap (\cup_{t\in[\tau]}I_2^D(t,\ell'))|$ the number of common points between the node sets $\cup_{t\in[\tau]}I_2^{D}(t,\ell)$ and $\cup_{t\in[\tau]}I_2^{D}(t,\ell')$, one has:
\begin{equation}\label{eq_intersections}
 \forall \ell\ne \ell'\in[K/M],\quad \dE_1\left( |(\cup_{t\in[\tau]}I_2^{D}(t,\ell))\cap (\cup_{t\in[\tau]}I_2^D(t,\ell'))|\II_{\cE}\right) =O\left( \frac{D^2}{n} \right).
\end{equation}
\end{lemma}
The Lemma's proof idea is as follows. The $\tau$ non-overlapping alternative path segments, that we refer to as a $\tau$-path, are obtained by selecting uniformly at random one such $\tau$-path among all present in the graph. Then \eqref{eq_dvar_1} is established by showing that the number of $\tau$-paths concentrates. In turn, this concentration is established by bounding the variance of the number of $\tau$-paths. This is done using the Markov chain bounding technique used in Lemma \ref{lemma:mkv_upper_bound}.  The second part of the Lemma, \eqref{eq_intersections}, requires further concentration results on the numbers of $\tau$-paths, that follow from applying Janson's inequality \cite{boucheron2013concentration}, p. 205, Theorem 6.31.

The proof idea of Theorem \ref{thm:impossible_reconstr_line} (detailed in the appendix) is then as follows. The $\tau$-paths of Lemma \ref{dvar_control} provide $\tau$ alternative $K$-paths to the actual planted path. These are ``lures'' for the optimal MAP reconstruction algorithm, that must return on average as many points of each of these lure paths as of the planted path. Since all these $\tau+1$ paths have intersection of negligible size, the overlap achieved by MAP must necessarily be at most $K/(\tau+1)$.

\section{Proof strategy for planted $D$-ary trees}\label{sec:5}
We assume here that $\Gamma$ is a complete $D$-ary tree of size $K$ and depth $h$, with $D > 1$ a fixed constant.


Under $\dP_0$, the neighbourhood of a given vertex in $G$ is close to a Galton-Watson process with offspring law $\Poi(\lambda)$. The probability of the existence of an infinite $D$-ary subtree in this process is the largest non-negative root $p_*(D, \lambda)$ of the equation
\begin{equation}\label{eq:fixed_point}
p=\psi_D(\lambda p),
\end{equation}
where 
$$
\psi_D(\mu) := \dP(\Poi(\mu) \geq D),\; \mu\ge 0.
$$
The behavior of the random graph differs based on whether the above probability is zero or not. We define the \textit{critical} threshold $\lambda_D$ as
\begin{equation}\label{eq:def_lambda_D}
	\lambda_D = \sup\left\{ \lambda > 0 \ \big|\  p_*(D, \lambda) = 0 \right\}
\end{equation}
In the following, we  focus on \textit{subcritical} $\lambda$, that is whenever $\lambda < \lambda_D$.
\subsection{Study of the Galton-Watson process}\label{subsec:gw}
Let $(T, o)$ be a rooted Galton-Watson tree with offspring law $\Poi(\lambda)$, with $\lambda < \lambda_D$. The following Theorem characterizes the distribution of the maximum height of a $D$-ary tree rooted in $o$.
\begin{theorem}\label{thm:gw_height}
	Let $(T, o)$ be a Galton-Watson tree as above, and $n > 0$. Let $p_h$ be the probability that a $D$-ary tree of height $h$ rooted in $o$ is contained in $T$. Then, for almost all $\lambda$, there exists $h_*$ such that
	\begin{align}
		p_{h_* + 1} &= o\left(\frac1n\right) \\
		p_{h_*} &= \Omega(n^{-c}) \text{ for some } c < 1
	\end{align}
	Moreover, as $n \to\infty$ one has $h_* = \frac{\ln\ln(n)}{\ln(D)} + O(1)$.
%
\end{theorem}

Thus $h_*$ depends on $\lambda$ only through terms of lower (constant) order. The Theorem's proof, detailed in the appendix, relies on the following 
\begin{lemma}\label{lem:rec_relation}
 The sequence $p_h$ satisfies the  recurrence relation
	\begin{align*}
		p_1 &= 1 \\
		p_{h+1} &= \psi_D(\lambda p_h) \text{ for all } h \geq 1.
	\end{align*}
\end{lemma}
Necessarily $0 \leq p_{h+1} \leq p_h$ for all $h$ (since a tree of height $h+1$ contains a tree of height $h$), and therefore by continuity of $\psi_D$, $p_h$ converges as $h\to\infty$ to the largest fixed point of \eqref{eq:fixed_point}.
By definition of $\lambda_D$, the only solution of this equation is $p_\infty = 0$, and thus
\begin{equation}
	\lim_{h\to\infty} p_h = 0
\end{equation}
Now, $\psi_D(x) \sim \frac{x^D}{D!}$ as ${x \to 0}$, which
implies that for $h$ large enough, $p_{h+1} \simeq C\,p_h^D$, and thus
$p_h \simeq C\,\varepsilon^{D^h}$
for some small $\varepsilon > 0$. 
A more rigorous version of this argument, as well as its use in the proof of Theorem \ref{thm:gw_height}, is presented in the Appendix.

\subsection{Coupling and application to planted trees}
Following the insights from the previous section, we define the two thresholds $\overline h$ and $\underline h$ by :
\begin{align*}
	\overline h &= \inf \left\{ h > 0 \ \Big|\  p_h < \frac1n \right\},&
	\underline h = \sup \left\{ h > 0 \ \Big|\  p_h > \frac{\ln(n)}n \right\}\cdot
\end{align*}
Theorem \ref{thm:gw_height} implies that $\overline h \sim \frac{\ln\ln(n)}{\ln(D)}$, and that for almost all $\lambda$, $\overline h = \underline h + 1$, and otherwise $\overline h = \underline h + 2$. Also, $p_{\overline h} = o(\frac1n)$ and $p_{\underline h} = \Omega(n^{-c})$ for some $c > 1$. The following  Theorem connects the study from section \ref{subsec:gw} to our planted tree problem:

\begin{theorem}\label{thm:dary_height}
	Let $G$ be a graph drawn according to $\dP_0$, and $h > 0$.
	Then with high probability:

	1.  For $h \leq \underline h$,  there are $\omega(1)$ $D$-ary trees of height $h$ in $G$.
		
	2. For $h \geq \overline h + C$, where $C$ is a large enough constant,  there are no $D$-trees of height $h$ in $G$.
\end{theorem}

The second part of this theorem yields an easy detection algorithm whenever $h \geq \overline h + \Omega(1)$.

\begin{corollary}\label{cor:dary_detection}
	Assume that $\Gamma$ is a complete $D$-ary tree of height $h$, with $h \geq \overline h + \Omega(1)$. Then w.h.p under $\dP_0$, $X_{\Gamma} = 0$, and therefore the test $T(G) = 1$ iff $X_{\Gamma} > 0$ discriminates between $H_0$ and $H_1$ correctly with high probability.
\end{corollary}

The two statements of Theorem \ref{thm:dary_height} are a consequence of the following coupling lemma, whose proof, as well as the full proof of the theorem, is deferred to the appendix :

\begin{lemma}\label{lem:coupling}
	For a graph $G$ and a vertex $v$ in $G$, denote by $(G, v)_\ell$ the $\ell$-neighbourhood of $v$ in $G$. Similarly, let $(T, o)_\ell$ be the $\ell$-neighbourhood of $o$ in the Galton-Watson process described above.

	Then, under $\dP_0$, assuming that $\ell = o(\log(n))$, the total distance variation between the law of $(G, v)_\ell$ and that of $(T, o)_\ell$ goes to 0 as a negative power of $n$ when $n \to\infty$.

	\medskip

	Furthermore, for $\lambda' > \lambda$, and $(T', o')$ a GW process with parameter $\lambda'$, then, provided the $\ell$-neighbourhood of $v$ is cycle-free, there exists a coupling between $(G, v)_\ell$ and $(T', o')_\ell$ such that $(G, v)_\ell \subseteq (T', o')_\ell$ with probability 1.
\end{lemma}

There is therefore a sharp cutoff in the probability of presence of tree of height $h$ in $G$, and we have already seen in Corollary \ref{cor:dary_detection} that it can be leveraged to obtain a detection algorithm when $h \leq \underline h$. It remains however to study two aspects of the problem: reconstruction for $ h \geq \overline h$, as well as the possibility (or lack thereof) of detection when $h\leq \underline h$.

\subsection{Likelihood ratio and detection for $h \leq \underline h$}
We conjecture, as is the case when $D = 1$, that when $h = \underline h - \omega(1)$, then the total variation distance $|\dP_1 - \dP_0|_{\mathrm{var}}$ goes to $0$ when $n\to\infty$. However, the Markov chain bounds used for lines cannot be easily adapted to the current setting, and we only prove this result for $h \leq \underline h - \Omega(\ln\ln\ln(n))$ :

\begin{theorem}\label{thm:dtv_dary}
	Assume that $\Gamma$ is a $D$-ary tree of height $h$, with $D > 1$ and 
	\[ h \leq \underline h - \frac{\ln(\underline h)}{\ln(D)} + \frac{\ln\left(1 - \frac1D\right)}{\ln(D)}. \]
	Then, the total variation distance $|\dP_1 - \dP_0|_{\mathrm{var}}$ goes to zero as $n\to\infty$. Thus, for any test $T(G) \in \{0, 1\}$, $\dP_1(T(G) = 1) - \dP_0(T(G) = 1) \to 0$ as $n\to\infty$.
\end{theorem}

As before this is deduced from the following Lemma, shown in the Appendix:
\begin{lemma}\label{lem:likelihood_dary}
	Under the same assumptions as Theorem \ref{thm:dtv_dary}, $\dE_0(L^2) \to 1$ as $n \to \infty$.
\end{lemma}
We believe the following stronger version of the Theorem to hold: 
\begin{conjecture}
	The result of Theorem \ref{thm:dtv_dary} holds true for all $h \leq \underline h$.
\end{conjecture}

If true, this conjecture would complete the bottom left part of the phase diagram for $D$-ary tree, with a sharp threshold between undetectability and detection/reconstruction.

\subsection{Reconstruction for large $h$}

When $\lambda < \lambda_D$ and $h \geq \overline h$, we have shown that under $\dP_0$ there is w.h.p no copy of $\Gamma$ in $G$. One could therefore expect to be able to reconstruct $\Gamma$ with overlap $1 - o(1)$ ; however, this is not the case :

\begin{theorem}\label{thm:dary_norec}
	Given $\lambda > 0$, $h \geq \overline h $ such that $K = o(n)$, and a realization $G$ of the graph under $\dP_1$, the overlap achieved by any estimator $\hat{\mathcal K}$ of the attack is bounded above, i.e $\ov(\hat{\mathcal K}) \leq (1 - \delta)K$
	for some $\delta > 0$.
\end{theorem}
The proof is based on the fact that when $D > 1$, the leaves make up a positive proportion of $\Gamma$, and they are hard to reconstruct with high precision. On the other hand, since there is no copy of $\Gamma$ in $G$ w.h.p, one can still reasonably expect to achieve a partial reconstuction. This is the heuristic behind our second conjecture :
\begin{conjecture}
	For all $h \geq \overline h$, there exists a $\delta > 0$ and an estimator (possibly random) $\hat\cK$ such that w.h.p $ \ov(\hat\cK) \geq \delta K$. 
\end{conjecture}
\bibliography{graphical_combined,bib_2}
\appendix
\section{Proof of preliminary results}
\subsection{Proof of Lemma \ref{lem:likelihood_ratio}}
Let $\Gamma_1, \ldots, \Gamma_m$ be the copies of $\Gamma$ in $K_n$ the complete graph on $[n]$, where, denoting $\mathrm{Aut}(\Gamma)$ the automorphism group of $\Gamma$, $ m = \dbinom nK \frac{K!}{|\mathrm{Aut}(\Gamma)|}$. 
Then, by Bayes' formula, letting $e(G)$ denote the number of edges in graph $G$, one has for any graph $g$:
$$
\begin{array}{ll}
		\displaystyle\dP_1(G=g) = \frac1m\sum_{i=1}^m{\dP_0(G=g \,|\, \Gamma_i \in G)} 
				 &= \displaystyle \frac1m\sum_{i=1}^m{\II_{\Gamma_i \in g}\left(\frac\lambda n\right)^{e(g) - e(\Gamma_i)}\left(1-\frac\lambda n\right)^{\dbinom n 2 - e(g)}} \\
				 &\displaystyle = \frac1m\left(\frac \lambda n\right)^{-e(\Gamma)}\sum_{i=1}^m{\II_{\Gamma_i\in g}\,\dP_0(G=g)} \\
				 &\displaystyle = \frac{X_{\Gamma}}{\dE_0[X_{\Gamma}]} \dP_0(G),
\end{array}
$$
which completes the proof of Lemma \ref{lem:likelihood_ratio}.

\subsection{Proof of Theorem \ref{thm:star}}
	
We first prove that planted stars of size $K=\ln(n)/\ln(\ln(n))[1-\omega(1/\ln(\ln(n)))]$ are undetectable. The number $X$ of $K$-stars verifies 
$$
\dE_0(X)=n\binom{n-1}{K}\left(\frac{\lambda}{n}\right)^K\sim n\frac{\lambda^K}{K!}\cdot
$$
We will have undetectability if $\dE_0(L^2)\sim 1$, or equivalently by symmetry arguments, if
$$
\dE_0(X|\Gamma_1\in G)\sim \dE_0(X),
$$
where $\Gamma_1$ is an arbitrary $K$-star, e.g. that made of edges $(i,K+1)$, $i\in[K]$. We decompose $\dE_0(X|\Gamma_1\in G)$ into three terms $M_1$, $M_2$ and $M_3$, the expected numbers of $K$-stars centered respectively: at node $K+1$, at some  node $i\in[K]$, and finally at some node $i\in[n]\setminus[K+1]$. Since $M_3$ is upper-bounded by $\dE_0(X)$, it suffices to show that $M_1$ and $M_2$ are $o(\dE_0(X))$. One has:
$$
\begin{array}{ll}
M_2&=K\left( \binom{n-2}{K-1}\left(\frac{\lambda}{n}\right)^{K-1}+\binom{n-2}{K}\left(\frac{\lambda}{n}\right)^{K-1}\right)\\
&\le \frac{2K^2}{n}\dE_0(X)\\
&\ll \dE_0(X).
\end{array}
$$
Also,
$$
\begin{array}{ll}
M_1&=\sum_{\ell=0}^K\binom{K}{\ell}\binom{n-K-1}{K-\ell}\left(\frac{\lambda}{n}  \right)^{K-\ell}\\
&\le \sum_{\ell=0}^K \binom{K}{\ell}\frac{\lambda^\ell}{\ell!}\\
&\le (1+\lambda)^K.
\end{array}
$$
The desired result $M_1\ll \dE_0(X)$ will follow if
$$
\ln(n)+K\ln(\lambda)-\ln(K!)-K\ln(1+\lambda)\to +\infty.
$$
The terms in $K$ are of order at most $\ln(n)/\ln(\ln(n))$. By Stirling's formula, this will therefore hold provided $\ln(n)-K\ln(K)=\omega(\ln(n)/\ln(\ln(n)))$. By assumption,
$$
K\ln(K)\le \frac{\ln(n)}{\ln(\ln(n))}(1-\omega(1/\ln(\ln(n))))\ln(\ln(n))=\ln(n)-\omega(\ln(n)/\ln(\ln(n))),
$$
hence the undetectability result.

Similarly for detectability, the assumption that $K=\ln(n)/\ln(\ln(n))[1+\omega(1/\ln(\ln(n)))]$ entails that
$$
\ln (\dE_0(X))=K\ln(\lambda)+\ln(n)-\ln(K!)=-\omega(1).
$$
Thus a test which decides $H_1$ if and only if there is a node in $G$ with degree at least $K$ succeeds with high probability. Moreover, with high probability, only the centre of the planted star has degree at least $K$. The reconstruction method which consists in choosing, besides the highest degree node, $K$ of its neighbours chosen uniformly at random, achieves an overlap of $K-o(K)$: indeed, conditional on the planted star's centre having initially $Y$ neighbors in the original graph, the expected number of nodes in the reconstructed set will be 
$$
1+\frac{K^2}{Y+K}\ge 1+K(1-Y/K)=1+K-Y.
$$
Its expectation is lower-bounded by $K+1-\lambda$, and is thus $K-o(K)$.

\subsection{Proof of Theorem \ref{thm:easy_detect}}

Let $k$ denote the size of the hidden connected component, with $k=0$ under $\dP_0$ and $k=K$ under $\dP_1$. Let $A_1$ count the number of isolated nodes in $G$, $A_2$ the number of connected pairs $(i,j)$ that form an isolated component, and $A_3$ the number of triplets $(i,j,k)$ that form a connected component.

These quantities satisfy with high probability
\begin{equation}\label{eq:small_connected_counts}
A_1=e^{-\lambda}(n-k)+O(\sqrt{n}),\;A_2=\frac{(n-k)^2}{2}\frac{\lambda}{n}e^{-2\lambda}+O(\sqrt{n}),\;
A_3=\frac{(n-k)^3}{2}\frac{\lambda^2}{n^2}e^{-3\lambda}+O(\sqrt{n}).
\end{equation}
Indeed, only the $n-k$ nodes that are not part of the hidden connected graph can contribute to counts of connected components of size 1, 2 or 3. \eqref{eq:small_connected_counts} then follows from evaluation of the expectation and variance of these quantities.

Set $\hat{\lambda}=(nA_3)/(A_1 A_2)$. By \eqref{eq:small_connected_counts},
$\hat{\lambda}=\lambda+O(n^{-1/2})$. Now form 
$\hat{k}=n-e^{\hat{\lambda}}A_1$.
Again by \eqref{eq:small_connected_counts}, $
\hat{k}=n-(1-O(n^{-1/2}))(n-k)+O(\sqrt{n})=k+O(\sqrt{n})$.
Our test then decides $H_1$ if $\hat{k}\ge t_n$ and $H_0$ otherwise where $t_n$ is such that $\sqrt{n}\ll t_n \ll K$, which is indeed satisfied for $t_n=\sqrt{K\sqrt{n}}$). This ensures that the test discriminates correctly between the two hypotheses with high probability. Necessarily then, the variation distance $|\dP_0-\dP_1|_{var}$ goes to 1 as $n\to\infty$.

\subsection{Proof of Theorem~\ref{thm:reconstruction_fails}}

We first begin by a simple lemma, using the concentration of $X_\Gamma$ :
\begin{lemma}\label{lem:small_intersection}
    Let $\mathcal I_\Gamma$ be the proportion of pairs copies of $\Gamma$ in $G$ whose intersection is nonempty :
    \[ \mathcal I_\Gamma = \frac{1}{X_\Gamma^2}\sum_{\Gamma', \Gamma'' \in G}{\II_{\Gamma' \cap \Gamma'' \neq \emptyset}},\]
    where $\Gamma'$ and $\Gamma''$ range over all copies of $\Gamma$ in $G$.
    
    Then $\dE_0(\mathcal I_\Gamma) = o(1)$.
\end{lemma}

\begin{proof}(of Lemma~\ref{lem:small_intersection}).
    As in the proof of Lemma~\ref{lem:likelihood_ratio}, let $\Gamma_1, \dots, \Gamma_m$ be the copies of $\Gamma$ in $K_n$, and let $X_i = \II_{\Gamma_i \in G}$. Write 
	\begin{equation}
		\dE_0\left(X_{\Gamma}^2\right) = \sum_{i,j}\dE_0\left(X_iX_j\right) = \dE' + \dE'',
	\end{equation}
	where $\dE'$ is the sum over $\Gamma_i, \Gamma_j$ having disjoint vertex sets.
	
	We can easily compute $\dE'$ :
	\[ \dE' = \dbinom n K \dbinom{n-K}K \left(\frac{K!}{|\mathrm{Aut}(\Gamma)|}\right)^2 p^{2K-2} \sim \frac{n^{2K}p^{2K-2}}{|\mathrm{Aut}(\Gamma)|^2} \sim \dE_0\left(X_{\Gamma}\right)^2 \]
	Since $\dE_0\left(L^2\right) = 1 + o(1)$, it follows that
	\begin{equation}\label{eq:small_intersection}
	    \frac{\dE''}{\dE_0\left(X_{\Gamma}^2\right)} = o(1).
	\end{equation}
	
	Now, it is straightforward to see that
	\[ \sum_{\Gamma', \Gamma'' \in G}{\II_{\Gamma' \cap \Gamma'' \neq \emptyset}} = \sum_{\Gamma_i \cap \Gamma_j \neq \emptyset}X_iX_j. \]
	Recall that $L = X_\Gamma / \dE_0(X_\Gamma)$ ; we can decompose $\mathcal I_\Gamma$  as follows :
	
	\begin{align*}
	    \mathcal I_\Gamma &= \mathcal I_\Gamma \II_{L^2 > 1/2} + \mathcal I_\Gamma \II_{L^2 < 1/2}\\
	    &= \frac{\sum_{\Gamma_i \cap \Gamma_j \neq \emptyset}X_iX_j}{\dE_0\left(X_{\Gamma}^2\right)}\cdot \frac1{L^2} \cdot \II_{L^2 > 1/2} + \mathcal I_\Gamma \II_{L < 1/\sqrt{2}}
	\end{align*}
	
	We can now bound each term separately. The first one is straightforward since $1/L^2 < 2$ whenever the indicator variable is nonzero ; for the second one, notice that $I_\Gamma \leq 1$ and thus
	\begin{align*}
	    \dE_0(\mathcal I_\Gamma) &\leq \frac{\dE''}{\dE_0\left(X_{\Gamma}^2\right)}\cdot 2 + \dP_0\left(L < \frac1{\sqrt{2}}\right)\\
	    & = \frac{\dE''}{\dE_0\left(X_{\Gamma}^2\right)}\cdot 2 + o(1),
	\end{align*}
	having used the Bienaymé-Chebychev inequality to bound the second term.
	
	Using (\ref{eq:small_intersection}) then completes the proof.
\end{proof}

\vspace{1em}

We can now move on to the proof of Theorem~\ref{thm:reconstruction_fails} ; we first transform the expression of $\ov(\hat\cK)$ to better suit our needs :

\begin{align*}
    \mathrm{ov}(\hat\cK) &= \sum_G \sum_\cK \dP_1(G, \cK) \left|\hat\cK \cap \cK\right| \\
    &= \sum_G \dP_1(G) \sum_\cK \dP_1(\cK\, |\, G) \left|\hat\cK \cap \cK\right|
\end{align*}
where $\cK$ ranges over all $K$-subsets of $[n]$ and $G$ over all graphs on $n$ vertices.

The second sum can be transformed as in the proof of Lemma~\ref{lem:likelihood_ratio} into :
\begin{align*}
    \mathrm{ov}(\hat\cK) &= \sum_G \dP_1(G) \sum_{\Gamma' \in G} \frac{|\hat\cK \cap \Gamma'|}{X_\Gamma} \\
    &= \sum_G \dP_0(G) \sum_{\Gamma' \in G} \frac{|\hat\cK \cap \Gamma'|}{X_\Gamma} + o(K),
\end{align*}
since the conditions in Theorem~\ref{thm:reconstruction_fails} imply that $|\dP_1 - \dP_0|_{\mathrm{var}} = o(1)$ (see the remark after Theorem~\ref{thm:detection_fails}). The sum now ranges over all copies of $\Gamma$ in $G$.

This can now be expressed as an expectation :
\begin{align*}
    \mathrm{ov}(\hat\cK) &= \dE_0\left[ \sum_{\Gamma' \in G} \frac{|\hat\cK \cap \Gamma'|}{X_\Gamma} \right] + o(K) \\
    &= \sum_{i \in [n]}\dE_0\left[ \II_{i \in \hat\cK} \sum_{\Gamma' \in G} \frac{\II_{i \in \Gamma'}}{X_\Gamma} \right] + o(K).
\end{align*}

We can now finally use Lemma~\ref{lem:small_intersection} : indeed,
\begin{align*}
    \left(\sum_{\Gamma' \in G} \frac{\II_{i \in \Gamma'}}{X_\Gamma}\right)^2 &= \frac{1}{X_\Gamma^2}\sum_{\Gamma', \Gamma'' \in G}{\II_{i \in \Gamma'}\II_{i \in \Gamma''}} \\
    &\leq  \frac{1}{X_\Gamma^2}\sum_{\Gamma', \Gamma'' \in G}{\II_{\Gamma' \cap \Gamma'' \neq \emptyset}} \\
    &= \mathcal I_\Gamma.
\end{align*}

Therefore,
\begin{align*}
    \mathrm{ov}(\hat\cK) &\leq \sum_{i \in [n]}\dE_0\left[ \II_{i \in \hat\cK} \sqrt{\mathcal I_\Gamma} \right] + o(K) \\
    &= K\dE_0\left[ \sqrt{\mathcal I_\Gamma} \right] + o(K) \\
    &= o(K),
\end{align*}
using Jensen's inequality as well as Lemma~\ref{lem:small_intersection}. This completes the proof of Theorem~\ref{eq:small_intersection}.
\section{Detailed proofs for planted paths}\label{sec:appendix_lines}
\subsection{Proof of Lemma~\ref{lemma:LR_paths}} 
Expression \eqref{LR_ll} follows directly from Lemma \ref{lem:likelihood_ratio}. In the display below, by $\sum_{(i_1\cdots i_K)}$ we mean summation over all the $n(n-1)\cdots(n-K+1)$ oriented paths $(i_1,\ldots,i_K)$ of length $K$ over nodes in $[n]$. Write: 
$$
\begin{array}{lll}
\dE_0(L^2)&=&\sum_{(i_1\cdots i_K)}\sum_{(j_1\cdots j_K)}\left(\frac{(n/\lambda)^{K-1}}{n\cdots(n-K+1)}\right)^2 \dP_0(\hbox{paths }(i_1\cdots i_K)\hbox{ and }(j_1\cdots j_K)\hbox{ present in }G)\\
&=&\sum_{(i_1\cdots i_K)}\left(\frac{(n/\lambda)^{2(K-1)}}{n\cdots(n-K+1)}\right)\dP_0(\hbox{paths }(i_1\cdots i_K)\hbox{ and }(1\cdots K)\hbox{ present in }G)\\
&=&\left(\frac{n}{\lambda}\right)^{K-1}\dP_0(\hbox{ path }\pi=(I_1\cdots I_K)\hbox{ present in }G |\hbox{ path }(1\cdots K)\hbox{ present in }G),
\end{array}
$$
where $\pi=(I_1\cdots I_K)$ is a candidate path chosen uniformly at random from the $n(n-1)\ldots(n-K)$ possible length-$K$ paths. In the above we used symmetry to consider a single path $(1\cdots K)$ instead of all paths $(j_1\cdots j_K)$. 

Note that conditionally on the event that path $(1\cdots K)$ be present in $G$ and on the path $\pi$, the probability that path $\pi$ is also present in $G$ is given by $(\lambda/n)^{K-1-S}$, where $S$ is the number of edges in common between the two paths $\pi$ and $(1\cdots K)$. This yields expression \eqref{LR_lll}.
\subsection{Proof of Lemma \ref{lemma:mkv_upper_bound}}
Let $\cF_t=\sigma(I_1,\ldots I_t)$. Recall that $n'=n-K$. It is easily verified that we have the following inequalities for all $t=2,\ldots,K-1$:
$$
\dP(Z_t=1|\cF_t)\le\left\{
\begin{array}{ll}
\frac{1}{n'}&\hbox{if }Z_{t-1}=1,\\
\frac{2}{n'}&\hbox{if }Z_{t-1}=0,\\
0&\hbox{if }Z_{t-1}=-1.
\end{array}
\right.
$$
Similarly we have
$$
\dP(Z_t\ge 0|\cF_t)\le \frac{K}{n'}\cdot
$$
Moreover it is easily seen that $\dP(Z_1=1)\le(K/n')(2/n')$, and $\dP(Z_1\ge 0)\le K/n'$.

As in Lemma \ref{lemma:mkv_upper_bound}, we introduce the Markov chain $\{Z'_t\}_{t\ge 1}$ on state space $\{-1,0,1\}$ specified by the initial distribution
$\dP(Z'_1=1)=(K/n')(2/n')$, $\dP(Z'_1\ge 0)= K/n'$ and by the transition probability matrix $P$ in \eqref{eq:transition_proba_Z'}, that we recall for convenience:
$$
P=\left(
\begin{array}{lll}
1-K/n' & K/n'& 0\\
1-K/n' & (K-2)/n' & 2/n'\\
1-K/n' & (K-1)/n' & 1/n'
\end{array}
\right)
$$
The previous inequalities ensure that we can construct by induction over $t$ a coupled version of the two processes $\{Z_t\}$ and $\{Z'_t\}$ such that $Z_1\le Z'_1$, and for $t\ge 1$, if $Z'_t=-1$ then $Z_t=-1$, and furthermore we have the following implications:
$$
\begin{array}{ll}
Z_t=-1 &\Rightarrow Z_{t+1}\le Z'_{t+1},\\
Z_t=Z'_t&\Rightarrow Z_{t+1}\le Z'_{t+1},\\
(Z_t,Z'_t)=(1,0)&\Rightarrow Z_{t+1}\le Z'_{t+1}.
\end{array}
$$
Thus the only situation when we can have $Z_{t+1}>Z'_{t+1}$ is when $(Z_t,Z'_t)=(0,1)$. That is to say, for each time $t+1$ when process $Z$ hits 1 while chain $Z'$ does not, then at time $t$  chain $Z'$ hits 1 while process $Z$ does not.

Because of this, the number of times $t$ at which process $Z$ hits 1 is upper-bounded by the number of times $t$ at which chain $Z'$ does. Thus \eqref{eq:coupling1} holds, concluding the proof of Lemma \ref{lemma:mkv_upper_bound}. 

\subsection{Proof of Lemma \ref{lemma_1}}

By \eqref{eq:coupling1} and \eqref{squareLR}, $\dE_0(L^2)$ is upper bounded by 
\begin{equation}\label{eq:upper_bound_22}
\dE_0(L^2)\le \dE_0 x^{\sum_{s=1}^{K-1}Z'^+_s}.
\end{equation} 
To evaluate this term, introduce the row vector $F(t):=\{f_z(t)\}_{z\in\{-1,0,1\}}$ where $f_z(t):=\dE_0 x^{\sum_{s=1}^tZ'^+_s}\II_{Z'_t=z}$. We then have
\begin{equation}\label{eq:init}
F(1)=(\dP(Z'_1=-1),\dP(Z'_1=0),x\dP(Z'_1=1))=(1-K/n', K/n'(1-2/n'), x(K/n')(2/n')),
\end{equation}
together with the recurrence relation
\begin{equation}\label{eq:iter}
F(t+1)=F(t)M,
\end{equation}
where 
$$
M=\left(\begin{array}{lll}
1-K/n' & K/n' & 0\\
1-K/n' & K/n'-2/n' & x 2/n'\\
1-K/n' & K/n'-1/n' & x/n'
\end{array}
\right)
$$
Recall now that $x=n/\lambda$ and $n'=n-K$, so that $x/n'$ is asymptotic to $1/\lambda$. Thus the above matrix $M$ reads
$$
M=M_0+(K/n)M_1,
$$
where 
\begin{equation}\label{eq:M_0}
M_0=\left(
\begin{array}{lll}
1 & 0 & 0\\
1 & 0 & 2/\lambda\\
1 & 0 & 1/\lambda
\end{array}
\right),
\end{equation}
and the entries of matrix $M_1$ are $O(1)$. Note that $M_0$ admits eigenvalues $0, 1/\lambda,1$ with respective left eigenvectors
$$
\begin{array}{ll}
u_{0}&:=(1,1,-2),\\
u_{1/\lambda}&:=(-\lambda/(\lambda-1),0,1),\\
u_1&:=(1,0,0).
\end{array}
$$
We shall denote $(\mu_{r},v_r)$ the (eigenvalue,eigenvector) pair of  $M$ obtained by perturbation of the eigenpair $(r,u_r)$ of $M_0$, with $r\in\{0,1/\lambda,1\}$. By the Bauer-Fike theorem (see \cite{MR1477662}, Theorem VI.25.1),  $|\mu_r -r|=O(K/n)$ for all $r$.

Moreover Eq. (1.16), p. 67 in  \cite{Kato:1966:PTL} implies that a normed left (resp., right) eigenvector of $M$ associated to an eigenvalue $\mu_r$ of $M$ differs in norm from a normed left (resp., right) eigenvector of $M_0$ associated to eigenvalue $r$ by $O(K/n)$. We can thus chose $v_r=u_r+O(K/n)$.

Let the decomposition of vector $F(1)$ in the basis provided by the eigenvectors $\{v_r\}$ be given by:
$$
F(1)=\sum_{r\in\{0,1/\lambda,1\}}\alpha_r v_r.
$$
Denote by $e$ the all-ones $3\times 1$ column vector. The upper bound \eqref{eq:upper_bound_22} on $\dE_0(L^2)$ then gives
\begin{equation}\label{eqn_eig_ub}
\begin{array}{lll}
\dE_0(L^2)&\le & F(K-1) e\\
&=&
F(1)M^{K-2}e\\
&=&\sum_{r\in\{0,1/\lambda,1\}} \alpha_r v_r \mu_r^{K-2} e.
\end{array}
\end{equation}
By our choice of eigenvectors $v_r$ such that $|v_r-u_r|=O(K/n)$, and the fact that 
$$
F(1)=(1+O(K/n))u_{1}+O(K/n)u_{1/\lambda}+O(K/n)u_0,
$$
corresponding weights $\alpha_r$ verify $\alpha_1=1+O(K/n)$, $\alpha_{1/\lambda}=O(K/n)$, $\alpha_0=O(K/n)$. 

In the case where $\lambda>1$ and $K=o(\sqrt{n})$, \eqref{eqn_eig_ub} yields
$$
\dE_0(L^2)\le o(1)+(1+o(1))\mu_1^{K-2}=(1+o(1))(1+O(K/n))^{K-2}\le (1+o(1))e^{O(K^2/n)}.
$$
The assumption that $K=o(\sqrt{n})$ then allows to conclude. 

For $\lambda<1$ and $K=\ln(n)/\ln(1/\lambda) -\omega(\ln(\ln(n)))$, \eqref{eqn_eig_ub} yields
$$
\dE_0(L^2)\le (1+o(1))\left(1+O(K/n)\right)^{K-2}+O(K/n) \left(1/\lambda+O(K/n)\right)^{K-2}.
$$
The first term is $1+o(1))$ since $K^2/n=o(1)$. The second term's logarithm is equivalent to 
$$
\ln(K)-\ln(n)+(K-2)\ln(1/\lambda)\le \ln(\ln(n))-\ln(\ln(1/\lambda))-\omega(\ln(\ln(n))),
$$
and goes to $-\infty$ by assumption. 


\subsection{Proof of Lemma \ref{lemma:reconstr_lines_subcritical}}
We place ourselves under $\dP_1$ and condition on the fact that the $K$-path planted in the original Erd\H{o}s-Rényi graph $G_0$ is $k_1^K$. Denote for each $i\in[K]$ by $C_i$ the connected component of node $k_i$ in $G_0$.  Denote by $\cE_i$  the event that $C_i\cap \{\cup_{j\ne i}C_j\}\ne \emptyset$ and by $\cE'_i$ the event that $C_i$ contains a cycle. 

A standard construction of connected components based on a random walk exploration implies the existence of a constant $c>0$ such that for all $\ell\ge 0$, 
\begin{equation}\label{eq:ttt0}
\begin{array}{ll}
\dP(\cE'_i,|C_i|=\ell)&\le \frac{\lambda \ell^2}{n}\dP(|C_i|=\ell)\le \frac{\lambda \ell^2}{n} e^{-c\ell},\\
\dP(\cE_i,|C_i|=\ell)&\le \frac{\ell K}{n}e^{-c\ell},\\
\dP(|C_i|\ge \ell)&\le e^{-c\ell}.
\end{array}
\end{equation}
The first evaluation implies that with high probability, no $C_i$ contains a cycle (i.e. no $\cE'_i$ occurs) when $K=o(n)$. The second evaluation implies that the expected number of $i\in[K]$ such that $\cE_i$ occurs and $|C_i|\ge \ell$ is upper bounded, for some constant $c'>0$, by
$$
\sum_{i\in[K]}\dP(\cE_i,|C_i|\ge \ell)\le \frac{K^2}{n} e^{-c'\ell}.
$$
If $K^2=o(n)$, then this implies that with high probability, no $\cE_i$ occurs. Thus with high probability, there is no cycle in the connected component $C$. Moreover, the third evaluation in \eqref{eq:ttt0} ensures that 
$$
\sum_{i\in K}\dP(C_i\ge \sqrt{K})\le K e^{-c\sqrt{K}}=o(1).
$$
Thus the peeling process applied $\sqrt{K}$ times to $C$ returns exactly the planted $K$-path, except for $\sqrt{K}$ nodes at each of its ends.

If on the other hand, $K^2>o(n)$, we choose $\ell*=\theta\ln(n)$  and deduce from \eqref{eq:ttt0} that with probability $1-O(n^{-2})$, say, there is no $i\in[K]$ such that both $\cE_i$ and $|C_i|\ge \theta \ln(n)$ hold. The peeling process applied $\sqrt{K}$ times to $C$ then returns the planted path, shortened by no more than $\sqrt{K}$ nodes at each end, plus parts of the neighborhoods $C_i$ for which $\cE_i$ occurs. The expected number of nodes returned that do not belong to the planted path is therefore no more than
$$
K\dP(\cE_i)\ell^*=O(\frac{K^2}{n})\theta\ln(n).
$$
This is $o(K)$ under the assumption that $K=o(n/\ln(n))$. The conclusion of the Lemma follows.

\subsection{Proof of Theorem \ref{thm:impossible_reconstr_line}}
We show that Lemma \ref{dvar_control} implies \eqref{overlap_control_1}. First, the optimal overlap is achieved by the {\bf Maximum A Posteriori} (MAP) inference procedure, i.e. by putting in $\hat{\cK}$ the $K$ nodes with the highest probability, conditional on the observed graph $G$, of being in $\cK$. The probability that node $j$ belongs to $\cK$ conditional on $G$ is proportional to the number of $K$-paths in $G$ to which $j$ belongs. We denote by $\cK^*$ the corresponding set.  

Second, when under the alternative distribution $\dP_2:=\dP_0(G\in \cdot| k_1^K\in G, (k_i,i_2^{D},k_{i+D})\in G))$ in \eqref{eq_dvar_1}, the joint distribution of the numbers of $K$-paths going through the nodes $k_1^K$ or through the nodes 
in $k_1^i,i_2^{D},k_{i+D}^K$ are statistically indistinguishable. Thus, letting $N_{\ell}$ (respectively $N'_{\ell}$) denote the number of points of $k_{(\ell-1)M+1}^{\ell M}$ (respectively, $k_{(\ell-1)M+1}^{i},i_2^{D},k_{i+D}^{\ell M}$) that the MAP estimate selects, one has:
$$
\dE_2(N_{\ell})=\dE_2(N'_{\ell}).
$$
Let also $N'_{t,\ell}$ denote the number of points that the MAP estimate selects in $k_{(\ell-1)M+1}^{I(t,\ell)},I_2^D(t,\ell),k_{I(t,\ell)+D}^{\ell M}$. 
Since each of these variables is bounded by $M=L+D$, the variation distance bound \eqref{eq_dvar_1} implies
$$
\dE_1(N_{\ell})\le \dE_1(N'_{t,\ell})+\epsilon M.
$$
Summing these inequalities over $\ell\in[K/M]$ and $t\in[\tau]$ yields
\begin{equation}\label{eq_11}
\tau \sum_{\ell=1}^{K/M}\dE_1(N_\ell)=\tau \;\ov(\cK^*)\le\sum_{t=1}^\tau\sum_{\ell=1}^{K/M}\dE_1(N'_{t,\ell})+\epsilon \tau K.
\end{equation}
However, it holds that:
$$
\sum_{i=1}^K\II_{k_i\in\cK^*}+\sum_{j\in\cup_{t,\ell}I_{2}^{D}(t,\ell)}\II_{j\in \cK^*}\le K.
$$
This entails (using e.g. Bonferroni's inequality):
$$
\sum_{i=1}^K\II_{k_i\in\cK^*}+\sum_{\ell=1}^{K/M}\sum_{r=2}^{D}\sum_{t\in[\tau]}\II_{I_r(t,\ell)\in\cK^*} 
-\sum_{\ell\ne \ell',\ell,\ell'\in[K/M]} |(\cup_{t\in[\tau]}I_2^{D}(\ell))\cap(\cup_{t\in[\tau]}I_2^{D}(\ell'))|\le K.
$$
Taking expectations and using the last statement \eqref{eq_intersections} of the Lemma yields, separating evaluations on event $\cE$ and on its complementary set $\overline{\cE}$:
$$
\ov(\cK^*)+\left\{\sum_{t\in[\tau]}\sum_{\ell=1}^{K/M}\EE_1(N'_{t,\ell})\right\} -\tau L(K/M) -(K/M)^2 O(D^2/n) -\tau K \dP_1(\overline{\cE})\le K.
$$
Summed with the previous equation \eqref{eq_11}, this gives:
$$
(\tau+1) \ov(\cK^*)\le K + K\tau\left(\epsilon+(L/M) +(K/n)(D/M)^2+\dP_1(\overline{\cE})\right).
$$
The announced result follows from $\epsilon\ll 1$, $L\ll D$, $K=o(n)$  and $\dP_1(\overline{\cE})=o(1)$. 

\subsection{Proof of Lemma \ref{dvar_control}, Equation \eqref{eq_dvar_1}} 
We let $\pi_{i}$ denote the set of $\tau$ candidate paths $(k_{i},i_2^{D}(t,\ell),k_{i+D})_{t\in[\tau]}$ of the graph, where for fixed $\ell$, the $\{i_2^{D}(t,\ell)\}_{t\in[\tau]}$ are distinct and in $[n]\setminus k_1^K$. For $i\in[(\ell-1)M+1,\ell-1)M+L]$ these can all be used to construct the set of $\tau$ alternative paths in the $\ell$-th segment of $k_1^K$. We denote by
$$
\pi(\ell)=\cup_{i\in[(\ell-1)M+1,\ell-1)M+L]}\pi_{i}
$$
the corresponding collection. Our construction simply amounts to choosing a set of $\tau$ paths (that we shall call for short a $\tau$-path) uniformly at random from $\pi(\ell)$ in order to construct the alternative $\tau$-path for the $\ell$-th segment, and this independently for each segment. 

Denote $Z_i=|\pi_i|$. Then 
$$
\dE_1(Z_i)=(n-K)(n-K-1)\cdots(n-K-\tau(D-1)+1)\left(\frac{\lambda}{n}\right)^{\tau D}\sim \frac{1}{n^\tau}\lambda^{\tau D},
$$
since we assumed in \eqref{eq:LD_1} that $D\sim C \ln(n)$. Also, by symmetry,
$$
\begin{array}{lll}
\dE_1 Z_i^2 &=& \sum_{i_2^D(t),j_2^D(t)}\dP_1(\forall t\in[\tau],(k_i,i_2^D(t),k_i+D)\in G, (k_i,j_2^D(t),k_{i+D})\in G)\\
&=&\dE_1(Z_i)\sum_{j_2^D(t)}\dP_1(\forall t\in[\tau],(k_i,j_2^D(t),k_{i+D})\in G| \forall t\in[\tau],(k_i,i_2^D(t),k_{i+D})\in G),
\end{array}
$$
where in the last expression we fixed an arbitrary choice $(i_2^D(t))_{t\in[\tau]}$. It follows that:
$$
\dE_1 Z_i^2=(\dE_1(Z_i))^2\dE_1((n/\lambda)^S),
$$
where $S$ is the number of common edges between the fixed $\tau$-path $(k_i,i_2^D(t),k_{i+D})_{t\in[\tau]}$ and the $\tau$-path $(k_i,J_2^D(t),k_{i+D})_{t\in[\tau]}$ where $(J_2^D(t))_{t\in[\tau]}$ is chosen uniformly at random among $(\tau(D-1))$ sequences in $[n]\setminus k_1^K$. 

To control this second moment, we will condition on the number of common edges between each path $J_2^D(t)$ in the randomly selected $\tau$-path at its  beginning and  end with the beginning and end of some of the fixed paths $i_2^D(t')$, that we shall denote by $X_t$ and $Y_t$. These satisfy the constraints $X_t,Y_t\ge 0$, $X_t+Y_t\le D$. For $X_t+Y_t<D$, this forces the choice of $X_t+Y_t$ nodes among the $D-1$ to be chosen for path $J_2^D(t)$; for $X_t+Y_t=D$, this forces all the $D-1$ choices. Moreover, conditionally on $(X_t,Y_t)_{t\in[\tau]}$, the expectation of the variable $(n/\lambda)^S$ verifies
$$
\dE_1((n/\lambda)^S|(X_t,Y_t)_{t\in[\tau]})\le (n/\lambda)^{\sum_{t\in[\tau]}X_t+Y_t}(1+O(D/n))^{\tau D},
$$
by the Markov chain bounds in Lemma \ref{lemma:mkv_upper_bound}. By assumption, $D\ll \sqrt{n}$ so that $(1+O(D/n))^D=1+o(1)$. Thus, accounits for the $\tau!)^2$ choices of path correspondences between the beginnings and ends of the planted and random paths:
$$
\begin{array}{lll}
\dE_1 Z_i^2&\le&(\dE_1(Z_i))^2(\tau!)^2\left[(n/\lambda)^D n^{D-1}+\sum_{x,y\ge 0,x+y<D}(n/\lambda)^{x+y}n^{-(x+y)}(1+o(1))\right]^{\tau}\\
&\le&(\dE_1(Z_i))^2(1+o(1))(\tau!)^2[n \lambda^{-D}+(\sum_{x\ge 0}\lambda^{-x})^2]^2\\
&\le& (\dE_1(Z_i))^2(1+o(1))(\tau!)^2\left(\frac{\lambda}{\lambda -1}\right)^{2\tau}, 
\end{array}
$$
where we used that $n\lambda^{-D}=o(1)$. 

We now evaluate $\dE_1(Z_i Z_j)$ for $i\ne j$. The Markov chain bounding technique of Lemma \ref{lemma:mkv_upper_bound} directly applies to give:
$$
\dE_1(Z_i Z_j)\le (\dE(Z_i))^2(1+o(1)).
$$
Finally we obtain: 
$$
\begin{array}{lll}
\hbox{Var}(|\pi(\ell)|)&=&L \hbox{Var}(Z_i)+L(L-1)\hbox{Cov}(Z_i,Z_j)\\
&\le & \dE_1(Z_i)^2\left[L (1+o(1))(\tau!)^2\left(\frac{\lambda}{\lambda -1}\right)^{2\tau} +L^2 o(1)\right]\\
&\le & \dE_1(|\pi(\ell)|)^2 \left[\frac{O(1)}{L}+o(1)\right].
\end{array}
$$
Since by assumption $L\gg 1$, Tchebitchev's inequality implies that the random variable $|\pi(\ell)|$ concentrates: for some suitable $\epsilon=o(1)$, one has
$$
\dP_1\left(\left|\frac{|\pi(\ell)|}{\dE_1 |\pi(\ell)|}-1\right|\ge \epsilon   \right)\le \epsilon. 
$$
Denote by $\cA$ the event $\cA:=\{|\frac{|\pi(\ell)|}{\dE_1 |\pi(\ell)|}-1|\le \epsilon\}$. It thus has probability at least $1-\epsilon$. Consider a bounded function $f$ of the graph $G$. This concentration result allows us to establish the variation distance bound \eqref{eq_dvar_1} as follows. For some arbitrary candidate $\tau$-path $(i,i_2^D(t))_{t\in[\tau]}$, omitting for brevity the argument $t$ below, write:
$$
\dE_1(f(G)|\cA,\cK=k_1^K,I(\ell)=i,I_2^{D}(\ell)=i_2^{D})=\frac{\dE_1[f(G)\II_{\cA}\II_{(k_i, i_2^D,k_{i+D})\in G}\frac{1}{|\pi(\ell)|}]}{\dE_1(\II_{\cA}\II_{(k_i, i_2^D,k_{i+D})\in G}\frac{1}{|\pi(\ell)|})}.
$$
On $\cA$ one has
$$
\frac{1}{\dE_1|\pi(\ell)|}\frac{1}{1+\epsilon}\le \frac{1}{|\pi(\ell)|}\le \frac{1}{\dE_1|\pi(\ell)|}\frac{1}{1-\epsilon}.
$$
This yields:
$$
\frac{1-\epsilon}{1+\epsilon}\frac{\dE_1[f(G)\II_{\cA}\II_{(k_i, i_2^D,k_{i+D})\in G}]}{\dP_1((k_i, i_2^D,k_{i+D})\in G)} \le \dE_1(f(G)|\cA,\cK=k_1^K,I(\ell)=i,I_2^{D}(\ell)=i_2^{D})\le \frac{1+\epsilon}{1-\epsilon}\frac{\dE_1[f(G)\II_{(k_i, i_2^D,k_{i+D})\in G}]}{\dP_1(\cA\cap(k_i, i_2^D,k_{i+D})\in G)}.
$$
By symmetry over all $\tau$-paths in $\pi(\ell)$, denoting by $Z$ the total number of possible such $\tau$-paths in it ($Z\sim L n^{\tau(D-1)}$), one has 
$$
\dP_1(\cA\cap(k_i, i_2^D,k_{i+D})\in G)=\frac{1}{Z}\dE_1(|\pi(\ell)|\II_{\cA}). 
$$
However by definition of $\cA$ this is no smaller than 
$$
\frac{1}{Z}(1-\epsilon)\dE_1|\pi(\ell)| \dP_1(\cA)\: \ge \:(1-\epsilon)^2\dP_1((k_i, i_2^D,k_{i+D})\in G).
$$
Finally we obtain:
$$
\begin{array}{l}
\frac{1-\epsilon}{1+\epsilon}\left[\dE_1[f(G)|(k_i, i_2^D,k_{i+D})\in G]-||f||_{\infty}\epsilon\right] \le \dE_1(f(G)|\cA,\cK=k_1^K,I(\ell)=i,I_2^{D}(\ell)=i_2^{D})
\le \cdots \\
\cdots \le 
\frac{1+\epsilon}{(1-\epsilon)^3}\dE_1[f(G)|(k_i, i_2^D,k_{i+D})\in G].
\end{array}$$
The result of Equation \eqref{eq_dvar_1} follows.
\subsection{Proof of of Lemma \ref{dvar_control}, Equation \eqref{eq_intersections}} 
We define the event $\cE$ as, for some suitable constant $\alpha=\Omega(1)$:
\begin{equation}\label{eq_def_cE}
\cE:=\cap_{\ell\in[K/M]}\cE_{\ell},\hbox{ where }\cE_{\ell}:=\{|\pi(\ell)|\ge \alpha \EE_1|\pi(\ell)|\}.
\end{equation}
In the below display we let $I_2^D(\ell)=\cup_{t\in [\tau]}I_2^D(t,\ell)$, and $I_2^D(\ell)\cap I_2^D(\ell')$ the intersection of the two corresponding sets of nodes.
We then have for arbitrary $\ell\ne \ell'\in[K/M]$:
$$
\dE_1(|I_2^D(\ell)\cap I_2^D(\ell')|\II_{\cE})
=
\sum_{i}\sum_{j}\dE_1\left(\frac{1}{|\pi(\ell)|\cdot|\pi(\ell')|}\sum_{i_2^D\in\pi_i}\sum_{j_2^D\in \pi_j}|i_2^D\cap j_2^D|\II_{\cE}\right)
$$
where the first summations are over  $i\in[M(\ell-1)+1,M(\ell-1)+L]$ and $j\in [M(\ell'-1)+1,M(\ell'-1)+L]$. The expectation in the right-hand side does not depend on $i$ and $j$, by symmetry. Moreover, on $\cE$ we can upper bound the fraction in the expectation by $1/(\alpha\dE_1|\pi(\ell)|)^2$. Thus fixing some arbitrary $i\ne j$:
$$
\begin{array}{ll}
\dE_1(|I_2^D(\ell)\cap I_2^D(\ell')|\II_{\cE})&\le \frac{L^2}{(\alpha\dE_1|\pi(\ell)|)^2}\dE_1\left(\sum_{i_2^D\in\pi_i}\sum_{j_2^D\in \pi_j}|i_2^D\cap j_2^D|\right)\\
&\le 
 \frac{L^2}{(\alpha\dE_1|\pi(\ell)|)^2}\sum_{i_2^D, j_2^D}\dE_1\left(\II_{(k_i,i_2^D,k_{i+D})\in G}\II_{(k_j,j_2^D,k_{j+D})\in G}|i_2^D\cap j_2^D|\right),
\end{array}
$$
where summation is over all pairs of lists $i_2^D$ and $j_2^D$ of $\tau(D-2)$ distinct elements in $[n]\setminus k_1^K$. Denote by $J_2^D$ one such list selected uniformly at random, and by $i_2^D$ a fixed, arbitrary choice of one such list. One then has, recalling the expression of $\dE_1|\pi(\ell)|=L(\lambda/n)^{\tau D} (n-K)\cdots(n-K-\tau(D-1)+1)$:
\begin{equation}\label{eq:paths_nodes_inter_1}
\dE_1(|I_2^D(\ell)\cap I_2^D(\ell')|\II_{\cE})\le  \frac{1}{\alpha^2}\dE_1\left(\left(\frac{n}{\lambda}\right)^S|i_2^D\cap J_2^D|\right),
\end{equation}
where $S$ denotes the number of edges  in common between the two $\tau$-paths $i_2^D$ and $J_2^D$.

As in Lemma \ref{lemma:mkv_upper_bound}, we now define the Markov chain $\{Z'_t\}_{t\ge 0}$ on the three states $\{-1,0,1\}$, with transition probabilities given by the matrix
$$
P:=\left( 
\begin{array}{lll}
1-D/n' & D/n' & 0\\
1-D/n' & (D-2)/n' & 2/n'\\
1-D/n' & (D-1)/n' & 1/n'
\end{array}
\right),
$$
where $n'=n-K-D$, and with initial condition $Z'_0=-1$. These states are interpreted as follows: $Z'_t=-1$ if $J_{t+1}\notin i_2^D$, $Z'_t=0$ if $J_t\notin i_2^D$ and $J_{t+1}\in i_2^D$, and $Z'_t=1$ if $J_t,J_{t+1}\in i_2^D$. The same coupling argument as for Lemma \ref{lemma:mkv_upper_bound} implies, letting $x=n/\lambda$, the following, where the subscript in the second expectation specifies the initial state of the Markov chain $\{Z'_t\}$:
$$
\dE_1\left(\left(\frac{n}{\lambda}\right)^S|i_2^D\cap J_2^D|\right)\le \dE_{-1}\left( x^{\sum_{i=1}^{\tau(D-1)}Z_i'^+}\sum_{j=1}^{\tau(D-1)}\II_{Z'_j\ge 0}  \right).
$$
We introduce the notation $F_z(t)=(F_{z,-1}(t),F_{z,0}(t),F_{z,1}(t))$, where 
$$
F_{z,y}(t):=\dE_z\left( x^{\sum_{s=1}^tZ'^+_s}\II_{Z'_t=y}  \right). 
$$
It readily follows that 
$$
F_z(t)=(\II_{z=-1},\II_{z=0},\II_{z=1}) M^t,
$$
where 
$$
M:=\left( 
\begin{array}{lll}
1-D/n' & D/n' & 0\\
1-D/n' & (D-2)/n' & x*(2/n')\\
1-D/n' & (D-1)/n' & x/n'
\end{array}
\right).
$$
This matrix $M$ reads, as previously, $M_0+O(D/n)$ where $M_0$ is given by \eqref{eq:M_0}. 

Write then, using Markov's property:
$$
\begin{array}{ll}
\displaystyle \dE_{-1}\left( x^{\sum_{i=1}^{\tau(D-1)}Z'^+_i}\sum_{j=1}^{\tau(D-1)}\II_{Z'_j\ge 0}  \right)& =
\displaystyle \sum_{j=1}^{\tau(D-1)}\sum_{z\in\{0,1\}}\dE_{-1}\left( x^{\sum_{i=1}^{j}Z'^+_i}\II_{Z'_j=z}\right)\dE_{z}\left(x^{\sum_{i=1}^{\tau(D-1)-j}Z'^+_i}  \right)\\
&\displaystyle =\sum_{j=1}^{\tau(D-1)}\sum_{z\in\{0,1\}} F_{-1,z}(j)\sum_{y=-1,0,1}F_{z,y}(\tau(D-1)-j).
\end{array}
$$
Previously given perturbation results give the existence of coefficients $[\beta_{z,r}]_{z\in\{-1,0,1\},r\in\{0,1/\lambda,1\}}$ all in $O(1)$ such that
$$
F_z(0)=\sum_{r\in\{0,1/\lambda,1\}}\beta_{z,r} v_r.
$$
It follows that
$$
F_{z}(\tau(D-1)-j)=\sum_{r\in\{0,1/\lambda,1\}} \beta_{z,r} \mu_r^{\tau(D-1)-j} v_r=O(1),
$$
since $|\mu_r|\le 1+O(D/n)$ and $D^2\ll n$. It follows that
$$
\begin{array}{ll}
\dE_{-1}\left( x^{\sum_{i=1}^{\tau(D-1)}Z'^+_i}\sum_{j=1}^{\tau(D-1)}\II_{Z'_j\ge 0}  \right)&
\displaystyle =\sum_{j=1}^{\tau(D-1)}F_{-1}(j)\left(\begin{array}{c}0\\1\\1\end{array}\right)\times O(1)\\
&\displaystyle =\sum_{j=1}^{\tau(D-1)} \sum_{r\in\{0,1/\lambda,1\}}\beta_{-1,r}\mu_r^j v_r\left(\begin{array}{c}0\\1\\1\end{array}\right)\times O(1) .
\end{array}
$$
Since $F_{-1}(0)=u_1$, it holds that $\beta_{-1,1}=1+O(D/n)$, and $\beta_{-1,r}=O(D/n)$ for $r=0,1/\lambda$. The terms with $r=0,1/\lambda$ in the previous expression thus contribute at most $O(D^2/n)$. The terms with $r=1$ give
$$
\sum_{j=1}^{\tau(D-1)}\beta_{-1,r}\mu_1^j v_1\left(\begin{array}{c}0\\1\\1\end{array}\right)\times O(1)=O(D^2/n),
$$
by using the fact that $v_1=(1,0,0)+O(D/n)$. 

It remains to prove that the event $\cE$ defined in \eqref{eq_def_cE} is such that $\dP_1(\cE)=1-o(1)$. It will suffice to prove that for all $\ell\in[K/M]$, $\dP_1(\cE_{\ell})\ge 1-o(M/K)$. To show this we shall leverage Janson's inequality, as described in \cite{boucheron2013concentration}, p.205, Theorem 6.31. Applied to the random variable $|\pi(\ell)|$, it guarantees that for all $0\le t\le \dE|\pi(\ell)|$ one has
\begin{equation}\label{eq:janson}
\dP_1(|\pi(\ell)|\le \dE|\pi(\ell)|-t)\le e^{-t^2/(2\Delta)},
\end{equation}
where $\Delta$ is the expected number of ordered pairs of $\tau$-paths $(P,Q)$ in $\pi(\ell)$ that share at least an edge. Paralleling our previous bound on the variance of $|\pi(\ell)|$, we distinguish the pairs of $\tau$-paths $(P,Q)$ according to whether they share the same starting point $i\in[(\ell-1)M+1,(\ell-1)M+L]$ or not to write $\Delta=\Delta_1+\Delta_2$, and obtain:
$$
\begin{array}{l}
\Delta_1\le L \frac{\lambda^{2D\tau}}{n^{2\tau}}(1+o(1))(\tau!)^2\left(\frac{\lambda}{\lambda-1}\right)^{2\tau2},\\
\Delta_2\le L^2 \frac{\lambda^{2D\tau}}{n^{2\tau}}O\left(\frac{D^2}{n}\right).
\end{array}
$$
We moreover have that $\dE|\pi(\ell)|\sim L\frac{\lambda^{D\tau}}{n^{\tau}}$, so that
$$
\frac{(\dE|\pi(\ell)|)^2}{\Delta}\ge \frac{\Omega(1)}{\frac{1}{L}+\frac{D^2}{n}}\cdot
$$
By our choices \eqref{eq:LD_1} for $L$ and $D$, this lower bound is also $\Omega(1)L=C\Omega(1)\ln(n)$. Taking $t=(1-\alpha)\dE|\pi(\ell)|$ for some $\alpha\in(0,1)$  in \eqref{eq:janson}, we obtain
$$
\dP_1\left(|\pi(\ell)|\le \alpha \dE|\pi(\ell)|\right)\le \exp(-\alpha^2 C\Omega(1)\ln(n)/2).
$$
It readily follows that, for sufficiently large $C$, this probability can be made $o(n^{-3})$ (say), which suffices to conclude the proof of the Lemma.

\section{Proofs for planted $D$-ary trees}\label{sec:appendix_dary}

\subsection{Proof of Lemma \ref{lem:rec_relation}}

\begin{proof}
	The property $p_1 = 1$ is trivial. For $h \geq 1$, let $Z ~ Poi(\lambda)$ be the number of children of the root $o$. Each of the $Z$ children has independently a probability $p_h$ of being the root of a $D$-ary tree of height $h$. Therefore, if we define $Z_h$ to be the number of such children, we have
	\[ \mathcal L(Z_h \,|\, Z) \sim \mathrm{Bin}(Z, p_h). \]

	By the splitting property of Poisson random variables, $Z_h$ follows the distribtution $\Poi(\lambda p_h)$. But $T$ contains a $D$-ary tree of height $h$ rooted in $o$ if and only if $Z_h \geq D$, and the lemma follows.
\end{proof}

\subsection{Proof of Theorem \ref{thm:gw_height}}

\begin{proof}
	Let $h_0 > 0$ to be fixed later on ; there exists $\kappa > 0$  such that
	\begin{equation}
		\frac{(\lambda x)^D}{D!} \leq \psi_D(\lambda x) \leq e^{\kappa (D-1) p_{h_0}}\frac{(\lambda x)^D}{D!}
	\end{equation}
	for all $x \leq \epsilon$. Therefore, for $h \geq h_0$, one has

	\begin{equation}\label{recurrence_ineq}
		D\ln(p_h) + (D-1)c_{\lambda, D} \leq \ln(p_{h+1}) \leq D\ln(p_h) + (D-1)(c_{\lambda, D} + \kappa\, p_{h_0}).
	\end{equation}

	Iterating inequality (\ref{recurrence_ineq}), we get that for all $h \geq 0$ :
	\begin{equation}\label{absolute_ineq}
		D^h\left(\ln(p_{h_0}) + c_{\lambda,D} \right) - c_{\lambda,D}  \leq \ln(p_{h + h_0}) \leq D^h\left(\ln(p_{h_0}) +  c_{\lambda,D} + \kappa\, p_{h_0} \right) -  c_{\lambda,D} - \kappa\, p_{h_0}
	\end{equation}

	Choose $h_0$ such that $\alpha := -(\ln(p_{h_0}) + c_{\lambda,D} + \kappa\, p_{h_0}) > 0$, and let $$h_* = \left\lfloor \frac{\ln\left(\frac{\ln(n)}\alpha\right)}{\ln(D)} \right\rfloor + h_0$$

	Then $h_* + 1 = \frac{\ln\left(\frac{\ln(n)}\alpha\right)}{\ln(D)} + h_0 + \delta$ for some $\delta > 0$. Thus, using (\ref{absolute_ineq}), we find
	\[ \ln(p_{h_* + 1}) \leq  - D^\delta\ln(n) -  c_{\lambda,D} - \kappa\, p_{h_0}\]
	which yields that $p_{h_* + 1} = o\left(\frac1n\right)$ as required.

	On the other hand, for almost all $\lambda$ there is a choice of $h_0$ such that
	\[ h_* <  \frac{\ln\left(\frac{\ln(n)}\alpha\right)}{\ln(D)} + h_0 - \ln\left(\frac\alpha{\ln(p_{h_0}) + c_{\lambda, D}} \right) \]
	by continuity of the right-hand side. Then, for some $\delta' > 0$, we have
	\[ \ln(p_{h_*}) \geq -D^{-\delta'}\ln(n) - c_{\lambda, D}  \]
	which implies the second result of Theorem \ref{thm:gw_height}.

\end{proof}

\subsection{Proof of Lemma \ref{lem:coupling}}

This lemma is a classical result in sparse random graph theory (see e.g.  \cite{DBLP:conf/focs/BordenaveLM15}) ; we reproduce it here for the sake of self-containedness. First, a result on the size of neighbourhoods in $G$ :

\begin{lemma}[Lemma 29 in \cite{DBLP:conf/focs/BordenaveLM15}]\label{lem:neighbourhood_size}
  For a vertex $v$ in $G$, let $S_t(v)$ denote the size of the $t$-neighbourhood of $v$. Then there exists a constant $C$ such that with high probability, for every vertex $v \in G$ and $t \geq 0$ :

  \[ S_t(v) \leq C\ln(n)\alpha^t \]
\end{lemma}

We'll also use a bound on the number of vertices whose neighbourhood contains a cycle ; its proof, as well as the preceding lemma, can be found in \cite{DBLP:conf/focs/BordenaveLM15}.

\begin{lemma}[Lemma 30 in \cite{DBLP:conf/focs/BordenaveLM15}]
  Assume that $\ell = o(\ln(n))$. Then w.h.p there are at most $\ln(n)\lambda^{2\ell}$ vertices whose $\ell$-neighbourhood contains a cycle. Moreover, with high probability the graph $G$ is $\ell$ tangle-free, i.e. no vertex has more than one cycle in its $\ell$-neighbourhood.
\end{lemma}

We can now prove the first part of our lemma : consider the classical breadth-first exploration process which starts with $A_0 = \{v\}$ and at step $t \leq 0$, considers (if possible) a vertex $v_t \in A_t$ at minimal distance from $v$ and reveals its neighbors $N_{t+1}$ in $[n] \setminus \bigcup_t A_t$. It then updates $A_{t+1}$ as $A_t \cup N_{t+1}$ and repeats the process. We denote by $\cF_t$ the filtration generated by $A_0, \ldots, A_t$.

\begin{proof}(First part of Lemma \ref{lem:coupling}). Let $\tau$ be the stopping time at which $(G, v)_\ell$ has been revealed.By the two previous lemmas, with probability at least $1 - c\lambda^{2\ell}/n$, the neighbourhood $(G, v)_\ell$ is a tree. Therefore, we can mirror the discovery process in $(T, o)$, where at each step we discover the children of $v_t$. To establish the desired coupling result, we then only need to focus on the number of children of each node.

Given $\cF_t$, the number of discovered neighbors $y_{t+1}$ of the node $v_t$ has distribution $\mathrm{Bin}(n_t, \lambda/n)$, where
\[ n_t = n - \sum_{s=0}^t{y_s} \]

Therefore, given $\cF_t$, the total variation distance between the number of children of $v_t$ in $(G, v)_\ell$ and in $(T, o)_\ell$ is
\[ \left|\mathrm{Bin}(n_t, \frac\lambda n) - \Poi(\lambda)\right|_\mathrm{var} \]

The Stein-Chen method (see for example \cite{BarbourChen05}) yields that
\[ \left|\mathrm{Bin}(n_t, \frac\lambda n) - \Poi\left(\lambda\frac{n_t}n\right)\right|_\mathrm{var} \leq \frac\lambda {n}, \]
and a classical bound for Poisson law (see again \cite{BarbourChen05}) that
\[ \left|\Poi\left(\lambda\frac{n_t}n\right) - \Poi(\lambda) \right|_\mathrm{var} \leq \lambda\left(1 - \frac{n_t}n\right) \]

From Lemma \ref{lem:neighbourhood_size}, we find that $n_t \geq n - C\ln(n) \lambda^\ell$ with probability greater than $1 - 1/n$, and thus 
\[ \left|P_{t+1} - Q_{t+1}\right|_\mathrm{var} \leq \frac\lambda n + \lambda \frac{C\ln(n)\lambda^\ell}n, \]

where $P_{t+1}$ is the distribution of $y_{t+1}$ given $\cF_t$ and $Q_{t+1}$ is a $\Poi(\lambda)$ random variable independent of $\cF_t$. This finishes the proof of the first part of the lemma.
\end{proof}

For the second part, note that there exists a coupling $(X, X')$ such that $X \sim \Poi(\lambda)$, $X' \sim \Poi(\lambda')$ and $X' > X$ a.s. (take for example $X' = X + Z$ where $Z \sim \Poi(\lambda' - \lambda)$).

The proof is then straightforward : for every vertex $v$, we produce a coupling between the exploration process of $(G, v)_\ell$ and $(T', o')_\ell$ such that at each step $t$, the number of neighbors $y_t$ of $v_t$ in $G$ is less than in $T'$.

\subsection{Proof of Theorem \ref{thm:dary_height}}

\begin{proof}
	We first apply the first part of Lemma \ref{lem:coupling} to $\ell = \underline h = O(\ln\ln(n))$. Then, for at least $n - O(\ln(n)^\alpha)$ vertices $v$ (for some $\alpha > 0$), there is a coupling between $(T, o)_{\underline h}$ and $(G, v)_{\underline h}$. Since in $(T, o)_{\underline h}$, there is a copy of $\Gamma$ in $(T, o)_{\underline h}$ with probability $\Omega(n^{-c})$. It follows that w.h.p there is $\omega(1)$ copies of $\Gamma$ in $G$.

	\medskip

	Now, assume that $h = \overline h + C$, where $C$ is large enough such that for some $\lambda' > \lambda$, there are no trees of height $h$ in $(T', o')$ with probability $1 - o(1/n)$.

	For every $v \in G$ such that the $h$-neighbourhood of $v$ is a tree, we can produce a coupling of $(G, v)_h$ and $(T', o')_h$ such that $(G, v)_h \subseteq (T', o')_h$ with probability 1. Thus, with high probability, no vertex whose $h$-neighbourhood is a tree contains a copy of $\Gamma$ in said neighbourhood.

	Assume now that there is one cycle in the $h$-neighbourhood of $v$. With high probability, there is only one cycle going through $v$ in the neighbourhood. Thus, there are only two vertices in the neighbors of $v$ whose offspring contains a cycle. With probability $1 - O(n^{-c})$, no other neighbour of $v$ is the root of a $D$-ary tree of height $h-1$. If $D > 2$, then there is no copy of $\Gamma$ rooted in $v$ ; if $D = 2$, then both neighbors of $v$ in the cycle must be roots of disjoints binary trees of size $h-1$, in which case the cycle edge does not help.

	To summarize, the probability of presence of a copy of $\Gamma$ rooted at $v$ is upper bounded by $o(1/n)$ if the $h$-neighbourhood of $v$ is cycle-free, and by $O(n^{-c})$ if it is not. Since there are $O(\ln(n)^\alpha)$ such vertices, w.h.p there is no copy of $\Gamma$ in $G$.
\end{proof}

\subsection{Proof of Lemma \ref{lem:likelihood_dary}}

\begin{proof}
  In view of Lemma~\ref{lem:likelihood_ratio}, we aim to bound the ratio 
  $$\dE_0(L^2) = \frac{\dE_0(X_\Gamma^2)}{\dE_0(X_\Gamma)^2}$$.
  As before, let $\Gamma_1, \ldots, \Gamma_m$ be the copies of $\Gamma$ in the complete graph $K_n$, and let $X_i = \II_{\Gamma_i \in G}$. 

	We follow the proof sketch from \cite{bollobas2001random} : write 
	\begin{equation}
		\dE_0\left(X_{\Gamma}^2\right) = \sum_{i,j}\dE_0\left(X_iX_j\right) = \dE' + \dE'',
	\end{equation}
	where $\dE'$ is the sum over $\Gamma_i, \Gamma_j$ having disjoint vertex sets.
	
	We can easily compute $\dE'$ :
	\[ \dE' = \dbinom n K \dbinom{n-K}K \left(\frac{K!}{|\mathrm{Aut}(\Gamma)|}\right)^2 p^{2K-2} \sim \frac{n^{2K}p^{2K-2}}{|\mathrm{Aut}(\Gamma)|^2} \sim \dE_0\left(X_{\Gamma}\right)^2  \]

	We therefore need to show that $\dE'' = o\left(\dE_0\left(X_{\Gamma}\right)^2\right)$ ; to this end, note that if $\Gamma_i$ and $\Gamma_j$ are such that $v(\Gamma_i\cup \Gamma_j) = s$, then $e(\Gamma_i\cap \Gamma_j) \leq 2K - s - 1$ (since $\Gamma_i \cap \Gamma_j$ is a forest of size $2K - s$) and therefore $e(\Gamma_i \cup \Gamma_j) \geq s - 1$.

	Grouping the terms of $\dE''$ by the size of $\Gamma_i \cup \Gamma_j$, we get

	\begin{align*}
		\dE'' &\leq \sum_{s=K}^{2K-1}{\dbinom ns \dbinom{s}{s-K, s-K, 2K-s}\left(\frac{K!}{|\mathrm{Aut}(\Gamma)|}\right)^2\left(\frac\lambda n\right)^{s-1}} \\
			&= \frac{n}{\lambda|\mathrm{Aut}(\Gamma)|^2}\sum_{s=K}^{2K-1}{\frac{n^{\underline s}\,\lambda^{s}}{n^s}\frac{K!^2}{(s-K)!^{\,2}(2K-s)!}}\\
			&= \frac{n}{\lambda|\mathrm{Aut}(\Gamma)|^2}\sum_{s=K}^{2K-1}{\lambda^{s}\frac{K!^2}{(s-K)!^{\,2}(2K-s)!}\left(1 + O\left(\frac{K^2}n\right)\right)}\\
			&\leq \frac{n\lambda^{K-1}(1 + o(1))}{|\mathrm{Aut}(\Gamma)|^2}\sum_{u=0}^{K-1}{\lambda^u\frac{K!^2}{u!^2(K-u)!}},
	\end{align*}
	where we made the change of variables $u = s-K$. Now, write
	\[ \frac{K!^2}{u!^2(K-u)!} = \dbinom K u \frac{K!}{u!} \leq \dbinom K u K^{K-u}, \]
	and we get 
	\begin{align*}
		\dE'' &\leq \frac{n\lambda^{K-1}K^K}{|\mathrm{Aut}(\Gamma)|^2}(1 + o(1))\sum_{u=0}^{K-1}{\dbinom K u\left(\frac\lambda K\right)^u} \\
		&\leq \frac{n\lambda^{K-1}K^K}{|\mathrm{Aut}(\Gamma)|^2}(1 + o(1))\left(1 + \frac\lambda K\right)^K \\
		&\leq \frac{n\lambda^{K-1}K^K e^{\lambda}}{|\mathrm{Aut}(\Gamma)|^2}(1 + o(1)) \\
		&= O\left(\dE_0(X_{\Gamma})^2\times \frac{K^K}{n\lambda^K}\right)
	\end{align*}

  When $K \leq \frac{\ln(n)}{\ln\ln(n)}$, we find that $\dE'' = o\left(\dE_0[X_{\Gamma}]^2\right)$, as requested. But $K = \frac{D^{h+1} - 1}{D-1} \leq \frac{\ln(n)}{\ln\ln(n)}$ whenever 
  \[ h \leq \underline h - \frac{\ln(\underline h)}{\ln(D)} + \frac{\ln\left(1 - \frac1D\right)}{\ln(D)}, \]

  which is the condition mentioned in Theorem \ref{thm:dtv_dary}.
	
\end{proof}

\subsection{Proof of Theorem \ref{thm:dary_norec}}

\begin{proof}
	For $0 \leq p \leq h$, let $L_{p}$ be the set of vertices at depth $p$ of $\Gamma$, and $T_p$ the set of vertices at depth $\leq p$.

	The strategy of proof is as follows : we aim to prove that there exists a universal constant $\delta$ such that given $G$ and 
	\[ \mathcal T := \sigma(T_{h-1}) \subset G, \]
	the location of the first $h-1$ rows of $\Gamma$, we have with high probability on $G$
	\begin{equation}\label{eq:cond_overlap}
		\dP_1\left((\ov(\hat \cK) \leq (1 - \delta)K\ \Big|\ G, \mathcal T \right) = 1 - o(1)
	\end{equation}
	In what follows, we will consider $\mathcal T$ to be fixed, and $G$ drawn under $\dP_1$.

	\medskip

	Let $\varepsilon > 0$ to be adapted later, and consider two cases :

	\begin{itemize}
		\item $|\hat{\cK}\cap \mathcal{T}| \leq (1-\varepsilon)|\mathcal{T}|$ : in this case, we easily get
		\begin{align*} 
			\ov(\hat{\cK}) &\leq D^h + (1-\varepsilon)\frac{D^h - 1}{D-1} \\
			&= K - \varepsilon\frac{D^h - 1}{D-1} \\
			&= K - \varepsilon\frac{K-1}D \\
			&= (1 - \frac\varepsilon D)K + o(K),
		\end{align*}
		from which equation (\ref{eq:cond_overlap}) follows since $\varepsilon$ is independent from $G$ and $\mathcal T$.

		\item if $|\hat{\cK}\cap \mathcal{T}| > (1-\varepsilon)|\mathcal{T}|$, we need the following lemma :
	\end{itemize}

	\begin{lemma}\label{hypergeo}
		Let $\sigma(L_{h-1}) = \{i_1, \ldots i_{D^{h-1}} \}$, and define $n_k = |\mathcal N(i_k)|$ and $m_k = |\hat{\cK}\cap \mathcal N(i_k)|$. Then
		\[ \dE_1\left(|\hat{\cK} \cap \sigma(L_h)|\ \Big\vert\ G, \mathcal{T}\right) = D\sum_k{\frac{m_k}{n_k}} \]
	\end{lemma}

	\begin{proof}(of lemma \ref{hypergeo}).
		Given $\mathcal T$, all vertices that are neighbours of a vertex in $\sigma(L_{h-1})$ are equally likely to belong to $\Gamma$, since all $D$-ary trees in $G$ have the same probability of generating $G$.

		Therefore, given $G,\  \sigma(L_{h-1}) = \{i_1, \ldots i_{D^{h-1}} \}$, the random variable $N_k = |\hat{\cK} \cap \cK \cap \mathcal N(i_k)|$ follows a hypergeometric law of parameters $(n_k, D, m_k)$. If follows that
		\[ \dE_1(N_k) = D\frac{m_k}{n_k} \]

		Now, with high probability the neighbourhoods $\mathcal N(i_k)$ are disjoint and the variables $N_k$ are thus independent. Since $|\hat{\cK} \cap \sigma(L_h)| = \sum_k N_k$ whenever the $\mathcal N(i_k)$ are disjoint, the lemma follows.
	\end{proof}

	We can now prove our main theorem : notice that $|\mathcal N(i_k)| \sim D + \Poi(\lambda)$ since $K = o(n)$, so w.h.p a proportion $\alpha$ (for a universal constant $\alpha$) of the $i_k$ are such that $|\mathcal N(i_k)| \geq D+1$. Moreover,

	\[ S := \sum_k{m_k} = K - |\hat{\cK}\cap \mathcal{T}| < D^h + \varepsilon|\mathcal{T}| = (1+\frac\varepsilon{D-1})D^{h} + o(D^h) \]

	Thus, $S \leq (1+ \varepsilon')D^h$ for some $\varepsilon' > 0$.

	Let $I_1$ be the set of indices such that $n_k = D$ ; we have
	\begin{align*} 
		\sum_k{\frac{m_k}{n_k}} &= \sum_{k\in I_1}\frac{m_k}{n_k} + \sum_{k\notin I_1}\frac{m_k}{n_k}	\\
		&\leq \sum_{k\in I_1}\frac{m_k}D + \sum_{k\notin I_1}\frac{m_k}{D+1}
	\end{align*}

	Let $S_1 = \sum_{k\in I_1}m_k$ ; we know that
	\[ S_1 \leq D|I_1| \leq D(1-\alpha)D^{h-1}, \]
since $m_k \leq n_k = D$ on $I_1$, which yields
	\begin{align*}
		\sum_k{\frac{m_k}{n_k}} &\leq \frac{S_1}D + \frac{S - S_1}{D+1} \\
		&= \frac S{D+1} + \frac{S_1}{D(D+1)} \\
		&\leq D^{h-1}\left((1+\varepsilon)\frac{D}{D+1} + (1 - \alpha)\frac{1}{D+1}\right)\\
		&\leq D^{h-1}\left(1 - \frac{\alpha - D\varepsilon}{D+1}\right)
	\end{align*}

	Choosing $\varepsilon$ such that $\alpha - D\varepsilon > 0$, we eventually find
	\begin{equation}
		\dE_1\left(|\hat{\cK} \cap L_h|\ \Big\vert\ G, \mathcal{T}\right) \leq (1-\gamma)D^h
	\end{equation}
	for some $\gamma > 0$.
	
	Finally, we can bound $\hat{\cK} \cap \cK$ :
	\begin{align*}
		\dE_1\left(|\hat\cK\cap \cK|\ \Big|\  G, \mathcal T\right) &\leq |\mathcal{T}| + \dE_1\left(|\hat{\cK} \cap L_h|\ \Big|\  G, \mathcal T\right) \\
		&\leq (1 - \gamma)D^h + |\mathcal{T}| \\
		&\leq K - \gamma D^h + o(D^h) \\
		&\leq (1 - \gamma\frac{D-1}D)K + o(K),
	\end{align*}
	which completes the proof of Theorem \ref{thm:dary_norec}.
\end{proof}

\end{document}